\documentclass[12pt,a4paper,reqno]{amsart}

\usepackage{comment}

\usepackage{amsmath,amsfonts,amssymb,mathrsfs,stackrel}
\usepackage[english]{babel}
\usepackage[utf8]{inputenc}
\usepackage[T1]{fontenc}
\usepackage{graphicx}
\usepackage{tabularx}
\usepackage{hyperref}
\usepackage{tikz}
\usetikzlibrary{arrows}
\usepackage{fancyhdr}
\usepackage{cleveref}

\usepackage{comment}
\usepackage[all]{xy}

\numberwithin{equation}{subsection}

\newtheorem{theorem}{Theorem}[section]
\newtheorem{proposition}[theorem]{Proposition}
\newtheorem{lemma}[theorem]{Lemma}
\newtheorem{definition}[theorem]{Definition}
\newtheorem{corollary}[theorem]{Corollary}

\newtheorem{remarque}[theorem]{Remark}

\newcommand{\R}{\mathbb{R}}

\newcommand{\Z}{\mathbb{Z}}

\newcommand{\Q}{\mathbb{Q}}

\newcommand{\F}{\mathbb{F}}
\newcommand{\C}{\mathbb{C}}
\newcommand{\N}{\mathbb{N}}

\newcommand{\V}{\mathbb{V}}

\newcommand{\Br}{\mathrm{Br}}
\newcommand{\cD}{\mathcal{D}}
\newcommand{\cZ}{\mathcal{Z}}
\newcommand{\cY}{\mathscr{S}}
\newcommand{\cS}{\mathcal{S}}
\newcommand{\cM}{\mathcal{M}}
\newcommand{\cB}{\mathcal{B}}

\newcommand{\bad}{\mathrm{bad}}
\newcommand{\good}{\mathrm{good}}

\newcommand{\crys}{\mathrm{crys}}
\newcommand{\scrX}{\mathscr{X}}
\newcommand{\fP}{\mathfrak{P}}

\newcommand\numberthis{\addtocounter{equation}{1}\tag{\theequation}}

\DeclareUnicodeCharacter{0177}{\^u}

\DeclareMathOperator{\Spec}{Spec}

\DeclareMathOperator{\CH}{CH}

\DeclareMathOperator{\rec}{rec}

\title[]{Vanishing of Brauer classes on K3 surfaces under reduction}

\author{Davesh Maulik}
\address{Department of Mathematics, Massachusetts Institute of Technology,  77 Massachusetts Avenue, Cambridge, MA 02139}
\email{maulik@mit.edu}

\author{Salim Tayou}
\address{Department of Mathematics, Dartmouth College, 29 N. Main Street, Hanover, NH 03755, USA}
\email{salim.tayou@dartmouth.edu}

\date\today

\begin{document}
\begin{abstract}

Given a Brauer class on a K3 surface defined over a number field, we prove that there exists infinitely many reductions where the Brauer class vanishes, under certain technical hypotheses, answering a question of Frei--Hassett--V\'arilly-Alvarado.
\end{abstract}

\maketitle

\setcounter{tocdepth}{1}
\tableofcontents


\section{Introduction}
Let $X$ be a K3 surface over a number field $K$ and let $\alpha\in\mathrm{Br}(X)$ be a Brauer class on $X$. Let $\mathscr{X}\rightarrow \mathscr{S}$ be a smooth projective model, where $\mathscr{S}\hookrightarrow \mathrm{Spec}(\mathcal{O}_K)$ is an open subset of the spectrum of the ring of integers $\mathcal{O}_K$. 

For a prime $\mathfrak{P}$ of $\mathscr{S}$ where $\alpha$ is unramified, we have a reduction of $\alpha$ in the Brauer group of the reduction $\mathscr{X}_{\mathfrak{P}}$ that we denote by:  \[\alpha_{\mathfrak{P}}\in \Br(\mathscr{X}_{\mathfrak{P}})~.\]

The present paper aims to understand the following locus: 
\[\mathcal{S}(X,\alpha)=\{\mathfrak{P}\in\mathscr{S}|\, \alpha_\mathfrak{P}=0\}~,\]
and which was first investigated by Frei, Hassett, and V\'arilly-Alvarado in \cite{frei-hassett-alvarado}. To state our main result, let $\sigma:K\hookrightarrow \C$ be a complex embedding and let $T(X_{\sigma}(\C))$ be the transcendental lattice of $X_\sigma(\C)$. Let $N$ be the product of primes above which $X$ has bad reduction. In this article, we prove the following result. 
\begin{theorem}\label{main}
   Assume that the rank of $\mathrm{T}(X_\sigma(\C))$ is different from $2,4$, and that the torsion order of $\alpha$ in $\mathrm{Br}(X_{\overline{K}})$ is coprime to the discriminant of $\mathrm{T}(X_\sigma(\C))$ and $N$. Then the set $\mathcal{S}(X,\alpha)$ is infinite.
\end{theorem}

\subsection{Prior work and applications}
The question of triviality of Brauer classes on smooth projective surfaces under reduction has been raised in \cite{frei-hassett-alvarado}. The authors proved {\it loc. cit.} that on a K3 surface, a Brauer class becomes trivial for a positive density of primes, when the following assumptions are satisfied: the endomorphism field $E$ of $\mathrm{T}(X_\sigma(\C))$ is totally real and $\dim_E(\mathrm{T}(X_\sigma(\C)))$ is odd. If these assumptions are not satisfied and if the Brauer class is transcendental, then as explained in the introduction of \cite{frei-hassett-alvarado},  the set in \Cref{main} has density zero, up to a finite extension of $K$, by a result of Charles \cite[Theorem 1-(1)]{charles-Picard-number}. Our result hence gives a fairly general answer to this question with no assumptions on the Hodge structure of $X$, see \cite[Remark 1.4]{frei-hassett-alvarado}. The technical conditions appearing in the theorem are artifacts of the proof and we explain their appearance in the strategy of the proof below.  
\medskip 

\Cref{main} has several applications to rationality problems of cubic fourfolds and derived equivalences of twisted K3 surfaces which have been developed in \cite{frei-hassett-alvarado}, see \S\S 1.2, 1.3. It has also the following application: 

\begin{corollary}\label{cor:splitting-elliptic-fibration}
    Let $X$ be an elliptic K3 surface over a number field $K$ which admits a multisection of degree coprime to the product of the discriminant of $\mathrm{T}(X_\sigma(\C))$ and the primes of bad reduction of $X$. Then there exists infinitely many primes $\mathfrak{P}$ where the elliptic fibration on $X_{\mathfrak{P}}$ admits a section. 
\end{corollary}

\Cref{main} admits also a natural formulation over the complex numbers and in this case it follows from the results of \cite[\S 17.3]{voisin}, which can be furthermore sharpened into an equidistribution type statement in the spirit of \cite{tayouequi,tayoutholozan}. We simply formulate the statement here. 
\begin{theorem}
Let $\mathscr{X}\rightarrow \mathscr S$ be a non-isotrivial smooth projective family of K3 surfaces over a complex quasi-projective algebraic variety $S$ and let $\alpha$ be a Brauer class on the generic fiber $\mathscr{X}_\eta$. Then the locus in $\mathscr S$ where the reduction of the Brauer class $\alpha$ vanishes is analytically dense and equidistributed with respect to the metric given by the Chern form of the Hodge bundle. 
\end{theorem}

\subsection{Strategy of the proof}
The proof of \Cref{main} relies on Arakelov intersection theory on integral models of toroidal compactifications of GSpin Shimura varieties as developed in \cite{sstt} and \cite{tayouboundary} and is local-global in nature. If we denote by $r$ the torsion order of $\alpha$ in $\mathrm{Br}(X_{\overline{K}})$, also called \emph{the geometric torsion order of $\alpha$}, then we interpret the locus $\cS(X,\alpha)$ as an intersection locus of $\mathscr{S}$ with a family of special divisors in a Shimura variety $\widetilde{\mathcal{M}}_r$ with level structure of level $r$. Following a method initiated by Charles in \cite{charles-exceptional-isogenies} and generalized in \cite{sstt,tayouboundary} (see also \cite{maulik-shankar-tang-K3,shankar-tang,maulik-shankar-tang} for the function field setting), we control the global and local intersection numbers of $\mathscr{S}$ with a sequence of special divisors indexed by integers $m$ at archimedean and non-archimedean places and compare the order of growths. If there were only finitely many primes where $\alpha$ vanishes, then this means that the intersection is supported at finitely many primes independent of $m$. As $m$ grows, we get a contradiction by comparing the order of growths of the local and the global estimates.  

In performing this strategy, many new difficulties arise compared to the previous work \cite{sstt,tayouboundary}. The Shimura variety $\widetilde{\mathcal{M}}_r$ that we work with has level structure $r$ and hence the results of \cite{howardmadapusi} do not apply directly at the primes that divide $r$. We have thus to construct a suitable integral model $\widetilde{\mathcal{M}}_r$ which is ad hoc for our purposes and which is defined by a normalization process, similar to the one considered in \cite{madapusitor}.  In particular, it does not admit a moduli interpretation over the primes that divide $r$. The assumption of $r$ being coprime to the discriminant of $\mathrm{T}(X_\sigma(\C))$ is crucial as it shows that $\widetilde{\mathcal{M}}_r$ admits a finite map to the integral model constructed by Kisin \cite{kisin} (see also \cite{madapusiintegral}), which is smooth at the primes dividing $r$. In particular, we obtain by pullback an abelian scheme over $\widetilde{\mathcal{M}}_r$ as well as a system of realizations ($\ell$-adic, crystalline, de Rham).  We use these to define special divisors and show their compatibility with the pullback of the divisors appearing in \cite{howardmadapusi}.

The second difficulty in our paper is to construct a suitable Borcherds product whose divisor only involve the special divisors that we have constructed. We construct such a Borcherds product and we show that \cite[Theorem A]{howardmadapusi} extends over primes dividing $r$.  This allow us to derive the global estimate on the intersection number of $\mathscr S$ with a well chosen sequence of special divisors. The assumption on $r$ is used again to compare the special divisors and the Borcherds products with those from the smooth integral model via a flatness argument. In our approach, we do not prove that the resulting generating series of special divisors is a modular form, a result not needed for our purposes and which requires a deeper understanding of the integral models with full level structure $r$. To obtain the local estimates, we use a compatibility result between our special divisors and the divisors defined in \cite{howardmadapusi} to bound the local contributions using the estimates already appearing in \cite{sstt,tayouboundary}. The additional assumptions on $r$ are used at this level to allow us to input the estimates from \cite{sstt} and \cite{tayouboundary}.


\subsection{Organization of the paper}
In \S 2, we explain how to reduce \Cref{main} to an intersection theoretic statement in GSpin Shimura varieties. In \S 3, we introduce GSpin Shimura varieties, their integral models, special divisors and toroidal compactifications. We construct suitable integral models with level structure at $r$ and use them to write down the local and global estimates needed from Arakelov theory. In \S 4 we prove the global estimate on the intersection number using a well-chosen Borcherds products. In \S 5 we estimate the archimedean contributions, and in \S 6 we estimate the non-archimedean contributions. 

\subsection{Acknowledgement}
We thank Ananth Shankar and Yunqing Tang for useful conversations and the referees for detailed comments that helped clarify some of our arguments. We thank Igor Dolgachev for pointing out the application in \Cref{cor:splitting-elliptic-fibration} and Daniel Loughran for a helpful correspondence. S.T. was supported by NSF grant DMS-2302388 and NSF grant DMS-2503815.

\section{K3 surfaces and Brauer classes: some reductions}
We prove in this section how to deduce \Cref{main} from \Cref{principal}. Then we explain the connection to special quasi-endomorphisms on Kuga-Satake abelian varieties. For a detailed discussion on Brauer classes and K3 surfaces, we refer to \cite[Chapter 18]{huybrechts} and \cite[\S 4]{frei-hassett-alvarado}. 
\subsection{Background results}

Let $X$ be a K3 surface over a number field $K$ and let $\alpha\in\mathrm{Br}(X)$. Let $\sigma: K\hookrightarrow \C$ be a complex embedding and let $(L,Q)$ be the transcendental lattice of the complex K3 surface $X_{\sigma}(\C)$. Then we have the following maps between the different Brauer groups: 
\[\mathrm{Br}(X)\rightarrow \mathrm{Br}(X_{\overline{K}})\simeq \mathrm{Br}(X_\sigma)\subset \mathrm{Br}^{an}(X_\sigma({\C}))~,\]
where $\mathrm{Br}^{an}(X_\sigma({\C}))$ is the analytic Brauer group of $X_\sigma(\C)$. By a theorem of Gabber, the Brauer group $\mathrm{Br}(X_{\overline{K}})$ is torsion and is in fact equal to the torsion part of the analytic group $\mathrm{Br}^{an}(X_\sigma({\C}))$. 

It follows that the image of the class $\alpha$ in $\mathrm{Br}^{an}(X_\sigma(\C))$ is torsion of order $r\geq 1$, which may be different from its torsion order in $\mathrm{Br}(X)$. We refer to $r$ as the \emph{geometric torsion order} of $\alpha$. 
\medskip 

The class $\alpha$ admits B-lifts to $H^2(X_\sigma(\C),\Q)$ that we now recall following \cite[Page 415]{huybrechts}. From the exponential exact sequence, we get the exact sequence
\[0\rightarrow H^2(X_\sigma(\C),\Z)/\mathrm{Pic}(X_\sigma(\C))\rightarrow H^2(X_\sigma(\C),\mathcal{O}_{X_\sigma(\C)})\rightarrow \mathrm{Br}^{an}(X_{\sigma}(\C))\rightarrow 0~.\]
Since the torsion part of $\mathrm{Br}^{an}(X_\sigma (\C))$ is equal to $\Br(X_\sigma)$, we get:
\begin{align*}
    0\rightarrow H^2(X_\sigma(\C),\Z)+\mathrm{NS}(X_\sigma(\C))_\Q\rightarrow H^2(X_\sigma(\C),\Q)\rightarrow \Br(X_\sigma)\rightarrow 0,
\end{align*}
yielding an isomorphism 
\[\mathrm{Hom}_\Z(L,\Q/\Z)\simeq \mathrm{Br}(X_\sigma)~.\]

In particular, the $r$-torsion sub-groups are isomorphic: 
\[\frac{1}{r}L^{\vee}/L^{\vee}\simeq \mathrm{Hom}_\Z(L,\frac{1}{r}\Z/\Z)\simeq \mathrm{Br}(X_\sigma)[r].\]

Let $\beta \in \frac{1}{r}L^{\vee}/L^{\vee}$ be a preimage of $\alpha$. To summarize, we have proven the following proposition. 
\begin{proposition}\label{lifts}
    Let $L$ be the transcendental lattice of $X_\sigma(\C)$ and let $\alpha\in\mathrm{Br}(X)$ be a Brauer class of geometric torsion order $r\geq 1$. Then there exists $\beta\in \frac{1}{r}L^{\vee}/L^\vee$ which corresponds to the image of $\alpha$ in $\mathrm{Br}(X_{\overline{K}})[r]$.
\end{proposition}

\subsection{Compatibility with reductions}
We keep the notations from the previous section and let $\mathscr{X}\rightarrow \mathscr{S}$ be a smooth projective model of $X$ where $ \mathscr{S}\hookrightarrow \mathrm{Spec}(\mathcal{O}_K)$ is a Zariski open subset. 
\medskip 

Let $\mathfrak{P}$ be a prime of good reduction for $X$, i.e., in $\mathscr{S}$, where the Brauer class $\alpha$ is unramified and has torsion $r$ coprime to the residual characteristic of $\mathfrak{P}$. This excludes only finitely many primes.  


The Kummer exact sequence yields the following commutative diagram, where the middle vertical arrow is an isomorphism by smooth base change theorem:
\begin{align*}
\xymatrix{
 0 \ar[r] & \mathrm{NS}(X_{\overline{K}})\otimes \Z/r\Z\ar[r]\ar[d] & H^{2}_{\textrm{ét}}(X_{\overline{K}},\mu_r)\ar[r]\ar[d] & \Br(X_{\overline{K}})[r]\ar[r]\ar[d] & 0\\
0\ar[r] & \mathrm{NS}(\scrX_{\overline{\fP}})\otimes \Z/r\Z\ar[r] & H^{2}_{\textrm{ét}}(\scrX_{\overline{\fP}},\mu_r)\ar[r] & \Br(\scrX_{\overline{\fP}})[r]\ar[r] & 0
}
\end{align*}

By Artin's comparison theorem \cite[XVI 4]{grothendiecksga4}, for every prime number $\ell$, we have a natural $\mathrm{Gal}(\overline{K}/K)$-module $L_{\Z_\ell}$ in $H^{2}_{\textrm{ét}}(X_{\overline{K}},\Z_\ell(1))$. By compatibility of Poincar\'e pairings in Betti and \'etale cohomology, the dual lattice of $L_{\Z_\ell}$ is equal to $L^{\vee}_{\Z_\ell}$. If $\ell$ does not divide the discriminant of $L$, then in fact $L$ is self-dual at $\ell$ and $L^{\vee}_{\Z_\ell}=L_{\Z_\ell}$. 
\medskip 

Let $\alpha\in\mathrm{Br}(X)$ be a Brauer class of geometric torsion order $r$. By \Cref{lifts}, there exists $\beta\in \frac{1}{r}L^{\vee}/L^{\vee}$ that lifts $\alpha$. For every prime number $\ell$, we let $\beta_\ell\in \frac{1}{r}L^{\vee}/L^{\vee}\otimes \Z_\ell $ denote the $\ell$-adic component of $\beta$.  If $\ell$ is coprime to $r$, then $\beta_\ell=0$ and if $\ell$ is coprime to the discriminant of $L$, then 
\[\beta_\ell\in \frac{1}{r}L^{\vee}/L^{\vee}\otimes \Z_\ell\simeq \frac{1}{r}L/L\otimes \Z_\ell\simeq L_{\Z_\ell}/rL_{\Z_\ell}~.\]

\Cref{main} is then a consequence of the following statement. 
\begin{theorem}\label{principal}
Assume that $r$ is coprime to the discriminant of $L$. Then there exist infinitely many prime ideals $\mathfrak{P}$ such that there exists $\lambda\in \mathrm{Pic}(\mathscr{X}_{\overline{\mathfrak{P}}})$ which satisfies the following: for every prime $\ell$ coprime to $\mathfrak{P}$, the image of $\lambda$ under the isomorphism 
\[ H^{2}_{\textrm{ét}}(\scrX_{\overline{\mathfrak{\fP}}},\Z_\ell(1))\simeq H^{2}_{\textrm{ét}}(X_{\overline{K}},\Z_\ell(1)),\]
lies in $L_{\Z_\ell}$ and the residue class of $\lambda$ in $L_{\Z_\ell}/rL_{\Z_\ell}$ is equal to $\beta_\ell$. 
\end{theorem}
\subsection{Proof of \Cref{main}}
Assuming \Cref{principal}, we will prove in this section \Cref{main}. Let $\mathfrak{P}$ be a prime ideal given by \Cref{principal}, and where $\alpha$ is unramified. Let $r$ be the geometric torsion order of $\alpha$, which is coprime to $p$, the residual characteristic of $\mathfrak{P}$. We have then the following diagram, where the middle vertical arrow is an isomorphism by proper and smooth base change theorem: 
\begin{align*}
\xymatrix{\oplus_{\ell  \vert r}L_{\Z_{\ell}} \subset \oplus_{\ell\vert r} H^{2}_{\textrm{ét}}(X_{\overline{K}},\Z_\ell(1)) \ar[r]\ar[d] & \Br({X}_{\overline{K}})[r]\ar[d]\ar[r]&0\\ 
(\lambda)_{\ell\vert r} \in\oplus_{\ell\vert r} H^{2}_{\textrm{ét}}(\mathscr{X}_{\overline{\mathfrak{P}}},\Z_\ell(1)) \ar[r] & \Br(\mathscr{X}_{\overline{\mathfrak{P}}})[r]\ar[r]& 0, 
}
\end{align*}

By construction, the image of $\lambda$ in $\Br({X}_{\overline{K}})[r]$ is equal to $\alpha$. By commutativity of the diagram, this implies that the image of $\lambda$ in 
$\Br(\mathscr{X}_{\overline{\mathfrak{P}}})[r]$ is equal to the reduction $\alpha_\mathfrak{P}$ of $\alpha$. Since $\lambda \in \mathrm{NS}(\mathscr{X}_{\overline{\mathfrak{P}}})$, we can conclude that $\alpha_\mathfrak{P}=0$ in $\Br(\mathscr{X}_{\overline{\mathfrak{P}}})[r]$. Finally, we use the following lemma, which is taken from \cite[Lemma 4.4]{frei-hassett-alvarado}, to conclude.
\begin{lemma}
We have $\alpha_{\mathfrak{P}}=0$ in $\Br(\mathscr{X}_{\mathfrak{P}})$. 
\end{lemma}

\subsection{Special endomorphisms on Kuga-Satake abelian varieties}
We explain in this section our strategy for proving \Cref{principal}. By \cite[Theorem 3]{madapusiperatate} (and \cite{kimmadapusipera,ito-ito-koshikawa} when the characteristic is equal to $2$), up to a harmless extension of the number field $K$, we can associate to $X$ an abelian variety $A$ defined over $K$,  the \emph{Kuga-Satake} abelian variety such that for any prime $\mathfrak{P}$ of good reduction for $X$ and $A$, the $\Z_\ell$ and crystalline realizations of the primitive cohomology of $\mathscr{X}_{\mathfrak{P}}$ embed in those of $\mathrm{End}(\mathcal{A}_{\mathfrak{P}})$.   

Let $\mathfrak{P}$ be a place where both $X$ and $A$ have good reduction. For every $\beta\in L/rL$, let 
\[V_\beta(\mathcal{A}_{\overline{\mathfrak{P}}})\subset\mathrm{End}(\mathcal{A}_{\overline{\mathfrak{P}}}),\]
be the groups of  special endomorphisms of $\mathcal{A}_{\overline{\mathfrak{P}}}$, as defined later in \Cref{quasi-endomorphisms,s:self-dual}. 

\begin{proposition}
Let $r \geq 1$ and let $\mathfrak{P}$ be a prime of good reduction of residual characteristic coprime to $r$. Then there exists $\beta\in L/rL$ such that $V_\beta(\mathcal{A}_{\overline{\mathfrak{P}}})$ is different from zero if and only if there exists $\lambda\in \mathrm{Pic}(\mathscr{X}_{\overline{\mathfrak{P}}})$ that satisfies the conditions of \Cref{principal}.
\end{proposition}
\begin{proof}
Let $\widetilde{\beta}\in \frac{1}{r}L/L$ be a lift of the class $\alpha$ as given by \Cref{lifts} and let $\beta\in L/rL$ its image after multiplication by $r$. 

From the properties of the Kuga-Satake abelian variety, we have an inclusion: 
\[V_\beta(\mathcal{A}_{\overline{\mathfrak{P}}})\hookrightarrow \mathrm{Pic}(\mathscr{X}_{\overline{\mathfrak{P}}}).\]
Then for any non-zero endomorphism $f\in V_\beta(\mathcal{A}_{\overline{\mathfrak{P}}})$, the class $\lambda=f$ gives the desired result. Indeed, by definition of special endomorphisms, for any prime $\ell$, we have an $\ell$-adic realization $f\in L_{\Z_\ell}$ which by definition has residue equal to $\beta$ in $L/rL$, hence it satisfies \Cref{principal}. 
\end{proof}

We conclude that \Cref{principal} is implied by the following statement which we will prove in \Cref{section-proof}. 
\begin{theorem}\label{main-kuga-satake}
    For $\beta\in L/rL$ as in the proof above, there exists infinitely many primes $\mathfrak{P}$ coprime to $r$ such that $A$ has good reduction at $\fP$ and $V_\beta(\mathcal{A}_{\overline{\fP}})\neq \{0\}$.
\end{theorem}

\section{GSpin Shimura varieties: integral models and Arakelov intersection theory}
We introduce in this section GSpin Shimura varieties, their integral models and their toroidal compactifications. Our main references are \cite{agmp-annals,madapusiintegral,howardmadapusi,madapusitor} to which we refer for more details. 

\subsection{GSpin Shimura varieties over $\Q$}
Let $(L,Q)$ be a quadratic even lattice of signature $(n,2)$, $n\geq 1$ and denote the bilinear form associated to $(L,Q)$ by:
\[(x\cdot y)=Q(x+y)-Q(x)-Q(y),\, \forall x,y\in L~.\]

We can construct a Shimura datum associated to $(L,Q)$ as follows: let $G=\mathrm{GSpin}(L_\Q)$ be the reductive algebraic group over $\Q$ of spinor similitudes and consider the Hermitian symmetric domain 
\[\cD=\{\omega\in\mathbb{P}(L_\C), (\omega\cdot\omega)=0, (\omega\cdot\overline{\omega})<0\}.\]

Then $(G,\cD)$ is a Hodge type Shimura datum with reflex field equal to $\Q$. For any choice of a compact open subgroup $K\subset G(\mathbb{A}_f)$, we get a Shimura variety defined over $\Q$ whose set of complex points is \[M(\C)=G(\Q)\backslash \cD\times G(\mathbb{A}_f)/K,\] 
and whose canonical model $M$ is a smooth Deligne-Mumford stack over $\Q$.

The choice of the lattice $(L,Q)$ specifies a particular open subgroup of $G(\mathbb{A}_f)$ defined as $K=C(L\otimes \widehat{\Z})\cap G(\mathbb{A}_f)$, where $C(L\otimes \widehat{\Z})$ is the $\widehat{\Z}$-Clifford algebra of $(L\otimes \widehat{\Z},Q)$. The group $K$ is the largest compact-open subgroup of $G(\mathbb{A}_f)$ that stabilizes $L\otimes_\Z\widehat{\Z}$ and acts trivially on $L^{\vee}/L$ where $L^{\vee}$ is the dual lattice of $L$ defined as: 
\[L^\vee=\{x\in L_\Q|\,\forall y\in L, (x\cdot y)\in\Z\}.\]

The Shimura variety $M$ is of Hodge type and carries a family of Kuga-Satake abelian varieties $A\xrightarrow{\pi} M$ whose relative cohomology can be understood in terms of algebraic representations of $G$ as follows. By construction, $G$ has an algebraic action by left multiplication on $C(V)$ where $V=L\otimes_\Z \Q$, and $C(V)$ is the Clifford algebra of $(V,Q)$. There is also an action of $G$ on $V$ via an algebraic group morphism $G\rightarrow \mathrm{SO}(V)$. Letting $H=C(V)$, then we have an inclusion $V\hookrightarrow \mathrm{End}_\Q(H)$ given by left multiplication and it is in fact a $G$-equivariant map. This yields filtered vector bundles with integrable connection on $M$, denoted $(\mathbb{V}_{dR},F^{\bullet}\mathbb{V}_{dR})$ and  $(\mathbb{H}_{dR},F^{\bullet}\mathbb{H}_{dR})$ related by a morphism of flat filtered vector bundles \[\mathbb{V}_{dR}\hookrightarrow \mathbb{H}_{dR}~.\]
The vector bundle $\mathbb{V}_{dR}$ is endowed with a bilinear form \[(\,\cdot\,):\mathbb{V}_{dR}\times \mathbb{V}_{dR}\rightarrow \mathcal{O}_M,\]
 for which the line bundle $\omega=F^1\mathbb{V}_{dR}$ is isotropic and $F^0\mathbb{V}_{dR}=(F^1\mathbb{V}_{dR})^\bot$.
Moreover, we have a canonical isomorphism of filtered vector bundles: 
 \[ \mathbb{H}_{dR}\simeq \underline{\mathrm{Hom}}(R^1\pi_*\Omega^{\bullet}_{A/M},\mathcal{O}_M),\]
see \cite[\S 4.1]{agmp-annals} for more details. 
\bigskip 

The constructions above are functorial in the following way: for any inclusion $(L_1,Q)\subseteq (L_2,Q)$ of quadratic lattices, then the previous discussion produces Shimura varieties $M_1$ and $M_2$ over $\Q$ which admit Kuga-Satake abelian schemes $A_i\rightarrow M_i$ and filtered vector bundles with integrable connections $(\mathbb{V}^i_{dR},F^{\bullet}\mathbb{V}^i_{dR})$ and $(\mathbb{H}^i_{dR},F^{\bullet}\mathbb{H}^i_{dR})$, for $i=1,2$. We have a finite morphism $\eta:M_1\rightarrow M_2$, which is  \'etale if $L_1$ has finite index in $L_2$. We also have morphism of Kuga-Satake abelian schemes 
\begin{center}
\begin{tikzpicture}[scale=1]
\node (s) at (0,0) {$A_1$};
\node (s1) at (2,0) {$\eta^{*} A_2$} ;
\node (d) at (1,-1) {$M_1$};
\draw[->,>=latex] (s)--(s1);
\draw[->,>=latex] (s)--(d);
\draw[->,>=latex] (s1)--(d);
\end{tikzpicture}
\end{center}
which is an isogeny in the finite index case, of degree a power of $|L_2/L_1|$. Finally, we have canonical isomorphisms of filtered vector bundles with integrable connections: 
\[\eta^{*}\mathbb{V}^2_{dR}\simeq \mathbb{V}^1_{dR},\quad\textrm{and}\quad  \eta^{*}\mathbb{H}^2_{dR}\simeq \mathbb{H}^1_{dR}~.\]  
\begin{remarque}
Since $A_1$ and $A_2$ are defined using the Clifford algebras of two different lattices, they are different as abelian schemes. Even in the case where $L_1=rL_2$, the isogeny $A_1\rightarrow \eta^* A_1$, which is induced from the map of Clifford algebras $C(L_1)\rightarrow C(L_2)$, is not the multiplication by a power of $r$ but rather the multiplication by different powers of $r$ on each degree of the Clifford algebra.
\end{remarque}

\subsubsection{Integral models and their compactifications}
We recall in this section the construction of integral models of GSpin Shimura varieties  following \cite[\S 6]{howardmadapusi}, \cite[\S\S 4.2, 4.3]{agmp-annals}, and their toroidal compactifications following \cite{howardmadapusi,madapusitor}. 
\medskip 

Let $p$ be a prime number. The lattice $L$ is said to be {\it maximal} at $p$ if $L\otimes\Z_p$ is a maximal lattice of $L\otimes \Q_p$ over which the quadratic form is $\Z_p$-valued. In particular, if the lattice $L$ is self-dual at $p$, then $L$ is maximal at $p$. We say that $L$ is maximal if it is maximal at all primes. 

Let $\Omega$ be the finite set of primes $p\in \Z$ at which the lattice $L_{\Z_p}$ is not maximal. Then by \cite[\S 6]{howardmadapusi}, there is a normal and flat integral model $\mathcal{M}\rightarrow \Z[\Omega^{-1}]$ with generic fiber $M$, which is a Deligne-Mumford stack and which enjoys the following properties:  
\begin{enumerate}
    \item The Kuga-Satake abelian scheme extends to an abelian scheme $\mathcal{A}\rightarrow \mathcal{M}$. 
    \item The line bundle $\omega=F^1\mathbb{V}_{dR}$ extends to a line bundle  $\boldsymbol{\omega}$ on $\mathcal{M}$.
    \item $\mathcal{M}$ is smooth at a prime $p$ if the lattice $(L,Q)$ is almost self-dual and regular if $p$ is odd, and $p^2$ does not divide the discriminant of $L$.    
\end{enumerate}
To explain the last condition, we say that $L$ is almost self-dual at $p$ if either $p$ is odd and $L$ is self-dual at $p$ or $p=2$ and $v_2(|L^{\vee}/L|)\leq 1$, where $v_2$ is the $2$-adic valuation. 

\subsubsection{Special divisors}\label{quasi-endomorphisms}
By \cite[\S 4.5]{agmp-annals}, for every scheme $S\rightarrow\mathcal{M}$, there is a functorial subspace 
\[V(\mathcal{A}_S)\subset \mathrm{End}(\mathcal{A}_S)_\Q\]
of {\it special quasi-endomorphisms}, the construction of which will be recalled in \Cref{level-structure}. The space $V(\mathcal{A}_S)$ is endowed with a positive definite quadratic form $Q$ such that $x\circ x=Q(x)\cdot \mathrm{Id}_{\mathcal{A}_S}$ for $x\in V(\mathcal{A}_S)$. One in fact can define for every $\beta\in L^{\vee}/L$, a subset \[V_\beta(\mathcal{A})\subset V(\mathcal{A}_S)\]
of special quasi-endomorphisms whose different cohomological realizations are prescribed by $\beta$, see \cite[P. 447]{agmp-annals}. We have then the following result which is \cite[Proposition 4.5.8]{agmp-annals}.

\begin{proposition}
For every $\beta\in L^{\vee}/L$, $m\in Q(\beta)+\Z$, there is a finite, unramified and relatively representable $\mathcal{M}$-stack whose functor of points assigns to every scheme $\mathcal{S}\rightarrow \mathcal{M}$ the set 
\[\mathcal{Z}(\beta,m)(\mathcal{S})=\{x\in V_\beta(\mathcal{A}_S)|\, Q(x)=m\}\]
\end{proposition}

By \cite[Proposition 2.4.3]{howardmadapusi-2}, $\mathcal{Z}(\beta,m)$ is a generalized Cartier divisor in the sense of \cite[Definition 2.4.1]{howardmadapusi-2} and can also be seen as Cartier divisor on $\mathcal{M}$ by \cite[Remark 2.4.2]{howardmadapusi-2}. We will henceforth refer to it as \textit{special divisor}. 

We can give an explicit description of the set of complex points of the special divisors as follows: in $\mathcal{M}(\C)$, a point $s\in \mathcal{M}(\C)$ can be lifted to a pair \[(h,g)\in\cD\times G(\mathbb{A}_f),\] and the group of special quasi-endomorphisms of $\mathcal{A}_s$ is canonically identified with \[\{x\in L_\Q|\,(x\cdot h)=0\}~.\] 
Then the special divisors are given, for every $\beta\in L^{\vee}/L$ and $m$, by the following double quotient
\[\mathcal{Z}(\beta,m)(\C)=G(\Q)\backslash\left( \bigcup_{\underset{Q(\lambda)=m}{\lambda\in g.(\beta+\widehat{L})}}\{(h,g)\in \cD\times G(\mathbb{A}_f), (h\cdot\lambda)=0\}\right)/K~.\]

\subsection{GSpin Shimura varieties with level structure}\label{level-structure}
We fix a maximal quadratic lattice $(L,Q)$ for the rest of the paper and let  $r\geq 1$. Consider the inclusion of quadratic lattices 
\[(r L,Q)\subset (L,Q).\] 

The discussion from the previous section applies to both lattices $(L,Q)$ and $(rL,Q)$ yielding normal flat integral models \[\mathcal{M}\rightarrow \mathrm{Spec}(\Z)~,\quad \textrm{and}\quad \mathcal{M}_r\rightarrow \mathrm{Spec}(\Z[\Omega^{-1}])\] of $M$ and $M_r$. Here  $\Omega$ is the set of primes where $rL$ is not maximal, i.e., the prime divisors of $r$. We have thus an abelian scheme $\mathcal{A}_r\rightarrow \mathcal{M}_r$, and a Hodge line bundle $\boldsymbol{\omega}_r$.  We also have a finite \'etale map $\eta:M_r\rightarrow M$ which extends to a finite map over $\Z[\Omega^{-1}]$ by \cite[Proposition 6.6.1]{howardmadapusi} that we still denote by 
\[\eta:\mathcal{M}_r\rightarrow \mathcal{M}_{\Z[\Omega^{-1}]}~,\]
and such that $\eta^{*}\boldsymbol{\omega}\simeq \boldsymbol{\omega}_r$.

The Kuga-Satake abelian scheme $\mathcal{A}\rightarrow \mathcal{M}$ pulls back to an abelian scheme $\eta^{*}\mathcal{A}$ on $\mathcal{M}_r$ with an isogeny 
\begin{center}
\begin{tikzpicture}[scale=1]
\node (s) at (0,0) {$\mathcal{A}_r$};
\node (s1) at (2,0) {$\eta^{*} \mathcal A$} ;
\node (d) at (1,-1) {$\mathcal M_r$};
\draw[->,>=latex] (s)--(s1);
\draw[->,>=latex] (s)--(d);
\draw[->,>=latex] (s1)--(d);
\end{tikzpicture}
\end{center}
which extends the isogeny over the generic fibers. 
\medskip 

The following lemma is an easy consequence of the construction of the module of special quasi-endomorphisms, see also \cite[Proposition 6.6.2, 6.6.3]{howardmadapusi} which refers to \cite[Proposition 2.6.4]{agmp-compositio} for the proof. To simplify notations, we will drop the index $r$ in the notation of special divisors in $\mathcal{M}_r$, as it will be clear from their coset in which space they live. 
\begin{lemma}\label{cartier-equality}
For every $\beta\in L^{\vee}/L$, $m\in Q(\beta)+\Z$, we have an equality of Cartier divisors: 
\[\eta^*\mathcal{Z}(\beta,m)= \bigsqcup_{\underset{\overline{\gamma}=\beta}{\gamma \in L^{\vee}/rL}}\mathcal{Z}(\gamma,m)~.\]
\end{lemma}



\begin{definition}
    Let $\widetilde{\mathcal{M}}_r$ be the normalization of $\mathcal{M}$ in $\mathcal{M}_{r}$. This a normal flat integral model over $\Z$ of $M_r$ extending $\mathcal{M}_r\rightarrow \mathrm{Spec}(\Z[\Omega^{-1}])$. 
\end{definition}

It follows from the definition that we have the following commutative diagram: 
\begin{align*}
\xymatrix{\mathrm{Spec}(\Z)\ar@{=}[d]&&\widetilde{\mathcal{M}}_r \ar[ll]\ar[rr]\ar[d]^{\eta} && \mathcal{M}_r \ar[d]^{\eta}\\ 
\mathrm{Spec}(\Z)&&\mathcal{M} \ar[ll]\ar[rr] && \mathcal{M}_{\Z[\Omega^{-1}]}~. 
}
\end{align*}

The Kuga-Satake abelian scheme $\mathcal{A}\rightarrow \mathcal{M}$ pulls back to an abelian scheme $\eta^{*}\mathcal{A}$ on $\widetilde{\mathcal{M}}_r$, and the line bundle $\boldsymbol{\omega}$ pulls-back to a line bundle $\eta^*\boldsymbol{\omega}$ on $\widetilde{\mathcal{M}}_r$ which extends $\boldsymbol{\omega}_r$. By abuse of notations, we still denote $\boldsymbol{\omega}_r$ this extension.

Our goal in the next section is to extend the Cartier divisors $\mathcal{Z}(\beta,m)\rightarrow \mathcal{M}_r$ to $\widetilde{\mathcal{M}}_r$ such that the extension has good moduli interpretation and \Cref{cartier-equality} still holds. We will work at each prime in $\Omega$ then glue the constructions.
\subsubsection{Almost self-dual case}\label{s:self-dual}
Let $p$ be a prime number dividing $r$, hence $p\in\Omega$. We make the additional assumption that the lattice $L$ is {\bf almost self-dual at $p$} as this will be satisfied in our applications. Then the level $K_P$ at $p$ is hyperspecial and the Shimura variety $\mathcal{M}_{(p)}$ is the smooth canonical model over $\Z_{(p)}$ constructed in \cite{kisin, madapusiintegral,kimmadapusipera}. Let $\pi:\mathcal{A}\rightarrow \mathcal{M}$ be the Kuga-Satake abelian scheme. For $\ell\neq p$, we have an inclusion of \'etale sheaf of $\Z_\ell$-modules 
\[\V_{\ell}\subset \mathrm{End}_{\Z_\ell}(\mathbb{H}_\ell)\]
where \[\mathbb{H}_\ell=H^{1}_{\textrm{\'et}}(\mathcal{A}/\mathcal{M}_{(p)},\Z_\ell)~.\]

We also have an inclusion of filtered vector bundles with integrable connections:
\[\mathbb{V}_{dR}\subset \mathrm{End}(\mathbb{H}_{dR})~,\]
where \[\mathbb{H}_{dR}=H^1_{dR}({\mathcal{A}}/\mathcal{M}_{(p)})~,\]
and a crystal of modules over the formal completion of $\mathcal{M}_{(p)}$ along the special fiber: 
\[\mathbb{V}_{\crys}\subset \mathrm{End}(\mathbb{H}_{\crys}),\]
where $\mathbb{H}_{\crys}=R^1\pi_*\mathcal{O}^{\crys}_{\mathcal{A}_{2,\F_p}}$. Moreover, the formal completion of the de Rham vector bundle $\mathbb{V}_{dR}$ with its integrable connection is isomorphic to $\mathbb{V}_{\crys}$. 
\medskip 

We recall now the construction of special divisors in $\mathcal{M}_{(p)}$. For any scheme $S$ over $\Z_{(p)}$, the module of special quasi-endomorphisms $V(\mathcal{A}_{S})_{\Z_{(p)}}$ is by definition the set of quasi-endomorphisms $x\in \mathrm{End}_{\Z_{(p)}}(\mathcal{A}_{S})$ such that: 
\begin{itemize}
    \item the de Rham realization $x_{dR}$ lies in $\V_{dR|_S}$ and
    \item the $\ell$-adic realization $x_\ell$ lies in $\V_{\ell|_S}\otimes\Q_\ell$ and
    \item the $p$-adic realization $x_p$ over the generic fiber $S_\Q$ lies in $\V_{p|_{S_\Q}}\otimes \Q_p,$ and 
    \item its crystalline realization $x_{\crys}$ lies in $\V_{crys|_{S_{\F_p}}}$.
\end{itemize}

Let $\beta\in L^{\vee}/L$ and let $\beta_\ell\in L^{\vee}/L\otimes \Z_\ell$ be its $\ell$-adic component for every prime $\ell$. For $\ell\neq p$, the local system $\V_{\Z_\ell}^{\vee}/\V_{\Z_{\ell}}$ is trivial on $\mathcal{M}_{(p)}$  and isomorphic to $\underline{L^{\vee}/L\otimes \Z_\ell}$. Thus we have a well defined subsheaf 
\[\beta_\ell+\V_{\ell}\subseteq \V_\ell^{\vee}~.\] 
We define: 
\[V_{\beta}(\mathcal{A}_{S})=\{x\in V(\mathcal{A}_{S})_{\Z_{(p)}},\, \forall\, \ell\neq p,\, x_\ell\in\beta_\ell+\V_{\ell|_S}, x_p\in\beta_p+\mathbb V_{p|_{S_\Q}} \}~.\]
Via the morphism \[\eta:\widetilde{\mathcal{M}}_{r,(p)}\rightarrow \mathcal{M}_{(p)}~,\] all the  above data pulls-back to $\widetilde{\mathcal{M}}_{r,(p)}$: we have hence $\ell$-adic sheaves $\eta^*\V_{\ell}$, a de Rham vector bundle $\eta^{*}\V_{dR}$ and a crystal $\eta^*\V_{\crys}$.

\medskip 

For any $\Z_{(p)}$-scheme $S\rightarrow \widetilde{\mathcal{M}}_{r,(p)}$, we define the group of special quasi-endomorphisms:
\[ V(\eta^{*}\mathcal{A}_{S})_{\Z_{(p)}}\subset \mathrm{End}(\eta^{*}\mathcal{A}_{S})_{\Z_{(p)}}\]
as the quasi-endomorphisms $f\in \mathrm{End}(\eta^{*}\mathcal{A}_{S})_{\Z_{(p)}}$ whose \'etale, de Rham and crystalline realizations lies in the subsheaves $\eta^{*}\mathbb{V}_{\ell|_S}\otimes\Q_\ell$, $\eta^{*}\mathbb{V}_{dR|_S}$, $\eta^{*}\mathbb{V}_{crys|_S}$. This is simply the pull-back of $V(\mathcal{A}_{S})_{\Z_{(p)}}$.
\medskip

For $\ell\neq p$, notice that the \'etale local system $\frac{1}{r}\eta^*\mathbb{V}_{\ell}^{\vee}/r\cdot\eta^*\mathbb{\V}_\ell$ is trivial on $\widetilde{\mathcal{M}}_{r,(p)}$ and isomorphic to $\underline{\frac{1}{r}L^{\vee}/rL \otimes \Z_\ell}$. Hence, given $\beta\in L/rL$ and $\beta_\ell$ its $\ell$-adic component, we have a well defined subsheaf 
\[\beta_\ell+{r}\cdot\eta^*\mathbb{V}_{\ell}~.\] 
We define then 
\[V_{\beta_\ell}(\eta^{*}\mathcal{A}_{S})=\{x\in V(\eta^{*}\mathcal{A}_{S})_{\Z_{(p)}} | \, x_\ell\in\beta_\ell+r\cdot\eta^*\V_\ell\},\]

and \[V_{\beta_p}(\eta^{*}\mathcal{A}_{S})=\{x\in V(\eta^{*}\mathcal{A}_{S})_{\Z_{(p)}} | \,x_p\in\beta_p+r\cdot\eta^*\V_{p,S_\Q},\, \textrm{and}\, x_{\mathrm{\crys}}\in \eta^*\V_{crys,S_{\F_p}}\}.\]

Finally, we define
\[V_{\beta}(\eta^{*}\mathcal{A}_{S})=\cap_{\ell}V_{\beta_\ell}(\eta^{*}\mathcal{A}_{S})\cap V_{\beta_p}(\eta^{*}\mathcal{A}_{S})~.\]

We define now a functor on $\Z_{(p)}$-schemes as follows: 
\[S\mapsto \mathcal{Z}(\beta,m)(S)=\{f\in V_\beta(\eta^{*}\mathcal{A}_{S}), f\circ f=m\cdot\mathrm{Id}_{\eta^{*}\mathcal{A}_{S}}\}~.\]

\begin{proposition}
Let $\beta\in L/rL$. Then the above functor is representable by a finite unramified $\widetilde{\mathcal{M}}_{r,(p)}$-stack which coincides over $\Q$ with $\mathcal{Z}(\beta,m)_\Q$.  
\end{proposition}
\begin{proof}
We give an argument inspired by Proposition 2.7.1 in \cite{agmp-compositio}. From the moduli interpretation, we notice that we have an isomorphism of $\widetilde{\mathcal{M}}_{r,(p)}$-stacks:
 \[\eta^{*}\mathcal{Z}(m)=\bigsqcup_{\beta\in L/rL}\mathcal{Z}(\beta,m)\]
which shows that each $\mathcal{Z}(\beta,m)$ can be viewed as an open and closed substack of  $\eta^{*}\mathcal{Z}(m)$. Since $\mathcal{Z}(m)$ is representable by a finite unramified ${\mathcal{M}}_{(p)}$-stack, we conclude that the same is true for $\mathcal{Z}(\beta,m)$ over $\widetilde{\mathcal{M}}_{r,(p)}$. 

One can also give an alternative proof directly by following the proof explained in \cite[Proposition 2.7.2]{agmp-compositio} as we already have an abelian scheme $\mathcal{A}\rightarrow \widetilde{\mathcal{M}}_{r,(p)}$ and so the maximality assumption used there is not needed.
\end{proof}

In particular, it results from the previous proposition that the  divisors $\mathcal{Z}(\beta,m)$ glues as a finite unramified $\widetilde{\mathcal{M}}_r$-stack over $\Z$ and \'etale locally it is a Cartier divisor on $\widetilde{\mathcal{M}}_r$. Moreover, we have an equality of Cartier divisors:  
\begin{align}\label{pull-special-divisors}
    \eta^{*}\mathcal{Z}(m)=\bigsqcup_{\beta\in L/rL}\mathcal{Z}(\beta,m),
\end{align}
valid over $\Z$, and which extends \Cref{cartier-equality}.

\begin{proposition}\label{fltaness-special-cycles}
Let $\beta\in L/rL$, $m\in Q(\beta)+r\Z$. Then the Cartier divisor $\mathcal{Z}(\beta,m)\rightarrow \widetilde{\mathcal{M}}_r$ is flat over $\Z_{(p)}$. 
\end{proposition}
\begin{proof}
We have the relation \[\eta^{*}\mathcal{Z}(m)=\bigsqcup_{\gamma\in L/rL}\mathcal{Z}(\gamma,m)~,\] $\mathcal{Z}(m)$ is flat over $\Z_{(p)}$ by the same argument as in \cite[Prop 5.21]{madapusiintegral}, hence has no vertical components. Since $\eta$ is a finite map, we conclude that none of the $\mathcal{Z}(\gamma,m)$ has vertical components and by the lemma below applied to the complete local ring at a point, they are flat over $\Z_{(p)}$. 
\end{proof}
\begin{lemma}\label{normality-flatness}
Let $R$ be a normal, local, flat $\Z_{(p)}$-algebra and let $a$ be a non-zero divisor. Then all the associated primes of $a$ have height $1$. In particular, if $\mathrm{div}(a)\subset\mathrm{Spec}(R)$ has no vertical components of $\mathrm{Spec}(R\otimes\F_p)$, then $\mathrm{div}(a)$ is flat over $\Z_{(p)}$.
\end{lemma}
\begin{proof}
    This lemma is similar to \cite[Lemma 7.2.4]{howardmadapusi} when $R$ is Cohen-Macaulay but since we only assume normality, we give a detailed proof. By Serre's normality criterion, for every ideal $\mathfrak{P}$ of height $\geq 2$, $R_\mathfrak{P}$ has depth at least $2$ and hence $\mathfrak{P}$ cannot be associated to $a$, as otherwise the depth of $R_{\mathfrak{P}}/aR_\mathfrak{P}$ would be $0$, which is not possible as 
    
    \[\mathrm{depth}(R_\mathfrak{P}/aR_\mathfrak{P})=\mathrm{depth}(R_\mathfrak{P})-1\geq 1~,\] by \cite[Lemma 10.72.7.]{stacks-project}.
    For the second part, to prove that $\mathrm{div}(a)$ is flat, it is enough to prove that it has no $p$-torsion. By assumption, $a$ is not contained in any minimal prime over $p$, which are the same as the associated primes by the above. Hence $a$ is not a zero divisor in $R/pR$, which is equivalent to $p$ not being a zero divisor in $R/aR$, since $R$ is local and normal.
\end{proof}

\subsection{Arithmetic Chow groups}
We introduce in this section Arakelov Chow groups following \cite{gilletsoule} and \cite{burgos}. 
For more details on this section, we also refer to \cite[\S 3.1]{sstt} and \cite[\S 3]{tayouboundary}.

Let $(rL,Q)\subset (L,Q)$ be an inclusion of quadratic lattices of signature $(n,2)$ as before, in particular $L$ is maximal with discriminant coprime to $r$. Let $\widetilde{\mathcal{M}}_r$, $\mathcal{M}$ be the normal integral models over $\Z$ of the GSpin Shimura varieties associated to $(rL,Q)$ and $(L,Q)$ constructed in the previous section. 

Let $\Sigma$ be a rational polyhedral $K_r$-admissible cone decomposition. By the main theorem of \cite[Theorem 1]{madapusitor}, $\widetilde{\mathcal{M}}_r$ has a toroidal compactifications $\widetilde{\mathcal{M}}_r^{\Sigma}$  which is proper, normal and flat over $\Z$. Over $\C$, it is compatible with the toroidal compactification of its complex fiber as constructed in \cite[Chapter III]{amrt}. Let $\widehat{\CH}^1(\widetilde{\cM}^{\Sigma}_r,\mathcal{D}_{pre})_{\Q}$ be the first arithmetic Chow group of prelog forms as defined in \cite[Definition 1.15]{burgos}.
\medskip

For any toroidal stratum representative $(\Xi,\sigma)$ of type III where $\sigma$ is a ray, let $\cB^{\Xi,\sigma}$ be the corresponding boundary divisor of $\widetilde{\mathcal{M}}_r^\Sigma$ and for $\Upsilon$ a toroidal stratum representative of type II, let $\cB^{\Upsilon}$ be the corresponding boundary divisor of type II. Then by \cite[Theorem 1]{madapusitor}, both $\cB^{\Xi,\sigma}$ and $\cB^{\Upsilon}$ are relative Cartier divisors over $\Z$, hence flat over $\Z$.

Let $\beta\in L/rL$ and $m\in\Z$. We have defined in the previous section a special divisor \[\mathcal{Z}(\beta,m)\rightarrow \widetilde{\mathcal{M}}^\Sigma_r,\]
and following \cite{bruinierzemel}, see also \cite[Theorem 1.2]{tayoumock}, we define a corrected divisor in $\widetilde{\mathcal{M}}^\Sigma_r$: 

\begin{align}\label{completed-divisor}
\mathcal{Z}^{tor}(\beta,m)=\mathcal{Z}(\beta,m)+\sum_{\Upsilon}\mu_\Upsilon(\beta,m)\cB^\Upsilon+\sum_{(\Xi,\sigma)}\mu_{\Xi,\sigma}(\beta,m)\cB^{\Xi,\sigma},    
\end{align}
where the coefficients $\mu_\Upsilon(\beta,m)$ and $\mu_{\Xi,\sigma}(\beta,m)$ are defined in \cite[Eqs (4.5.1), (4.6.1)]{tayouboundary}.

Following \cite{bruinier,bruinierzemel}, the divisors $\mathcal{Z}^{tor}(\beta,m)$ can be endowed with a Green function $\Phi_{\beta,m}$ such that the pair: 
\[\widehat{\mathcal{Z}}^{tor}(\beta,m)=(\mathcal{Z}^{tor}(\beta,m),\Phi_{\beta,m})\] 
is an element of $\widehat{\CH}^1(\widetilde{\cM}^{\Sigma}_r,\mathcal{D}_{pre})_{\Q}$.


The Hodge line bundle $\boldsymbol{\omega}_r$ has a canonical Hermitian metric with prelog singularities, the \textit{Petersson metric}, see \cite[Equation (4.2.3)]{howardmadapusi} for a definition. Hence it defines an element 
\[\widehat{\boldsymbol{\omega}}_r\in\widehat{\mathrm{CH}}^1(\widetilde{\mathcal{M}}_r^{\Sigma},\mathcal{D}_{pre})~.\] 
\subsubsection{Arithmetic height and main estimates}

Let $K$ be a number field and let 
\[\rho:\mathscr{S}=\Spec(\mathcal{O}_K)\rightarrow \widetilde{\mathcal{M}}_{r}^{\Sigma}\] be an $\mathcal{O}_K$-point. Then the height $h_{\widehat{\mathcal{Z}}(\beta,m)}(\mathscr S)$ of $\mathscr{S}$ with respect to $\widehat{\mathcal{Z}}(\beta,m)$ is defined as the image of $\widehat{\mathcal{Z}}(\beta,m)$ under the composition: 

\[ \widehat{\mathrm{CH}}^1(\widetilde{\mathcal{M}}_r^{\Sigma},\mathcal{D}_{pre})\xrightarrow{\rho^*}\widehat{\mathrm{CH}}^1(\mathscr{S})\xrightarrow{\widehat{\deg}} \R~.\]
It is given by, see \cite[Equation (3.1)]{sstt}: 
\[h_{\widehat{\mathcal{Z}}(\beta,m)}(\mathscr S)=\sum_{\mathfrak{P}\subset \mathcal{O}_K}(\mathcal{Z}^{tor}(\beta,m).\mathscr{S})_\mathfrak{P}\log|\mathcal{O}_K/\mathfrak{P}|+\sum_{x\in\mathscr{S}(\C)}\Phi_{\beta,m}(x)~,\]
where for a prime $\mathfrak{P}\subset \mathcal{O}_K$:
\[(\mathcal{Z}^{tor}(\beta,m).\mathscr{S})_\mathfrak{P}=\sum_{v\in \left(\mathscr{S}\times_{\widetilde{\mathcal{M}}^{\Sigma}_{r}}\mathcal{Z}^{tor}(\beta,m)\right)(\overline{\mathbb{F}}_{\mathfrak{P}})}\mathrm{length} \left(\mathcal{O}_{\mathcal{Z}^{tor}(\beta,m)\times_{\widetilde{\mathcal{M}}^{\Sigma}_{r}}\mathscr{S},v}\right)~,\]
and ${\mathbb{F}}_{\mathfrak{P}}$ is the residual field of ${\mathfrak{P}}$.

\subsection{Main estimates}
Let $\beta \in L/rL$ and let $m\in \Z$ represented by $\beta+rL$. 

Let $c(\beta,m)$ be the $(\beta,m)$-th Fourier coefficient of Eisentein series $E_{rL}$ as in \cite[Prop. 3.1, (3.3)]{bruinierintegrals}, see also \cite[\S 3.3]{sstt}. For $n\geq 3$, we have $|c(\beta,m)|=-c(\beta,m)\ll_\epsilon m^{\frac{n}{2}}$ along the integers $m$ representable by $\beta+rL$.  

Our first main result is the following global height bound. 
\begin{proposition}\label{global-bound}
    As $m\rightarrow \infty$ and represented by $\beta+rL$, we have 
    \[\sum_{\mathfrak{P}\subset \mathcal{O}_K}(\mathcal{Z}(\beta,m).\mathscr{S})_\mathfrak{P}\log|\mathcal{O}_K/\mathfrak{P}|+\sum_{x\in\mathscr{S}(\C)}\Phi_{\beta,m}(x)=O(c(\beta,m))~.\]
\end{proposition}

Recall from \cite[\S 6]{sstt} that for a subset $S\subset \N$, the logarithmic asymptotic density is defined as:
\[\limsup_{X\rightarrow \infty} \frac{\log(\left|\{s\in S| X\leq s<2X\}\right|)}{\log X}~.\]

The second main results are estimates in average of multiplicities at archimedean and non-archimedean places.

\begin{proposition}\label{local-infinite}
For every $x\in\mathscr{S}(\C)$, there is a decomposition: 
\[\Phi_{\beta,m}(x)=c(\beta,m)\log(m)+A(\beta,m)+o\left(c(\beta,m)\log(m)\right).\]
Moreover, there exists a subset $S_{\bad}\subset \Z_{>0}$ of logarithmic asymptotic density $0$ such that
\begin{equation*}
\lim_{\substack{m\to\infty\\m\not\in S_\bad}} \frac{A(\beta,m)}{m^{\frac{b}{2}}\log m}=0.   
\end{equation*}
\end{proposition}

Next, we have the estimate at the non-archimedean places. Let $N$ be the product of primes where $\rho$ intersects the boundary of $\widetilde{\mathcal{M}}_r^{\Sigma}$. 
\begin{proposition}\label{good-finite}
Given $D, X\in \Z_{>0}$, $D$ coprime to $N$, let $S_{D,X}$ denote the set \[\{m\in \Z_{>0}\mid X \leq m<2X, \sqrt{\frac{m}{D}}\in \Z,\, (m,N)=1\}.\]
For a fixed prime $\fP$ and a fixed $D$, we have
\[\sum_{m\in S_{D,X}}(\mathscr S . \cZ(\beta,m))_\fP=o(X^{\frac{b+1}{2}}\log X).\] 
\end{proposition}


\subsection{Proof of \Cref{main-kuga-satake}}\label{section-proof}
Assuming \Cref{global-bound,local-infinite,good-finite} from the previous section, we prove here \Cref{main-kuga-satake} which proves \Cref{principal} and hence \Cref{main}. 
\medskip 

Let $K$ a number field and fix an embedding $\sigma:K\hookrightarrow \C$. Let $X$ be a K3 surface over $K$, and let $(\mathrm{T}(X_\sigma(\C)),Q)$ be the transcendental lattice of $X_\sigma(\C)$. Let $\alpha\in\mathrm{Br}(X)$ be a Brauer class of geometric torsion order $r$, assumed to be coprime to the discriminant of $\mathrm{T}(X_\sigma(\C))$. Let $\beta\in \mathrm{T}(X_\sigma(\C))/r\mathrm{T}(X_\sigma(\C))$ be a lift (multiplied by $r$) given by \Cref{lifts} and let $A$ be the Kuga-Satake abelian variety associated to $X$. 

Let $\mathrm{T}(X_\sigma(\C))\subset L$ be a maximal lattice containing $\mathrm{T}(X_\sigma(\C))$. Then we can see $\beta\in L/rL$. We are now in the setup of the previous sections: let $\widetilde{\mathcal{M}}_r^{\Sigma}$ be the toroidal compactification of the integral model of the Shimura variety associated to $(rL,Q)$ constructed in the previous sections. The Kuga-Satake abelian variety defines a $K$-point in $\widetilde{\mathcal{M}}_r^{\Sigma}$ which extends to a morphism: 
\[\rho:\mathscr{S}\rightarrow \widetilde{\mathcal{M}}_r^{\Sigma},\]
where $\mathscr{S}=\Spec(\mathcal{O}_K)$.

Assume by contradiction that the conclusion of \Cref{main-kuga-satake} does not hold. Then there exists finitely many primes $\mathfrak{P}_1,\ldots,\mathfrak{P}_s$ such that for every $m\in \Z$, the support of the intersection of $\mathscr{S}$ and $\mathcal{Z}(\beta,m)$ is contained in $\{\mathfrak{P}_1,\ldots,\mathfrak{P}_s\}$.

By \Cref{local-infinite}, there exists a subset $S_{\bad}\subset \Z_{>0}$ of logarithmic asymptotic density zero such that outside $S_{\bad}$ we have:
\[\sum_{x\in\mathscr{S}(\C)}\Phi_{\beta,m}(x)\asymp c(\beta,m)\log(m)+o(c(\beta,m)\log(m))\asymp -|c(\beta,m)|\log(m).\]

Let \[S^{good}_{D,X}=\{m\in S_{D,X}| m\notin S_{bad},\, (m,N)=1,\, m\equiv Q(\beta)\pmod{r}\}.\] 
Since $r$ is coprime to $N$, the set $\{m\in S_{D,X}| (m,N)=1,\, m\equiv Q(\beta)\pmod{r}\}$ has asymptotic density $\frac{1}{2}$, and we see that $|S^{good}_{D,X}|\asymp X^{\frac{1}{2}}$. \Cref{representation} ensures that any $m\in S^{good}_{D,X}$ is representable by $\beta+rL$, hebce $|c(\beta,m)|\gg X^{\frac{n}{2}}$ for $m\in S^{good}_{D,X}$. Thus we get 
\begin{align}\label{contradiction}
    \sum_{m\in S^{good}_{D,X}}\sum_{x\in\mathscr{S}(\C)}\Phi_{\beta,m}(x)\gg X^{\frac{n+1}{2}}\log X.
\end{align}
On the other hand, by \Cref{good-finite,good-finite}, we get by summing over the finitely many places where either $\mathscr{S}$ intersects a $\cZ(\beta,m)$ or which are of bad reduction:
\begin{align}\label{contradiction-2}\sum_{m\in S^{\good}_{D,X}}(\mathscr S.\cZ(\beta,m))_{\mathfrak P}\log|\mathcal{O}_K/\mathfrak{P}|=o(X^{\frac{n+1}{2}}\log X).
\end{align}

The combination of \Cref{contradiction} and \Cref{contradiction-2} contradicts \Cref{global-bound}. This proves the desired result.
\medskip 

The rest of the paper is devoted to proving the main estimates in Section 3.4.
\begin{remarque}
    In the previous discussion, the assumption $(r,N)=1$ was used to produce infinitely many integers $m$ that satisfy $(m,N)=1$ and $m\equiv Q(\beta)\pmod{r}$, hence its appearance in our main theorem. In fact, if $(r,N)>1$, then we can still prove \Cref{main} under a weaker assumption as follows.  First, observe that the norm $Q(\alpha)$ of the class $\alpha$ is a well-defined element in $\Z/r\Z$ and is equal to $Q(\beta)$, for a lift $\beta$. If the reduction of $Q(\alpha)$ modulo $\gcd(r,N)$ is invertible,  then we can still find an infinite sequence of integers $m$ that satisfy $(m,N)=1$ and $m\equiv Q(\alpha)\pmod{r}$. This is enough to prove \Cref{main}.
\end{remarque}
\section{Global estimate}

We prove in this section \Cref{global-bound}. Our method is inspired from \cite{howardmadapusi} and relies on Fourier-Jacobi expansions of Borcherds products at the cusps of GSpin Shimura varieties.
\subsection{Background results}\label{background}
Let $(L,Q)$ be a quadratic lattice of signature $(n,2)$ and let $\mathcal{M}$ be the normal integral model over $\Z[\Omega^{-1}]$ of the GSpin Shimura variety associated to $(L,Q)$ constructed in \cite{howardmadapusi}. Let $f\in M^{!}_{1-\frac{n}{2}}(\rho_L)$ be a weakly holomorphic modular form of weight $1-\frac{n}{2}$ with respect to the conjugate Weil representation $\overline{\rho}_L$. We assume that the principal part of $f$ has integral coefficients and denote it by: 
\[\sum_{\beta\in L^\vee/L}\sum_{\underset{m<0}{m\in -Q(\beta)+\Z}}c(\beta,m)q^m,\,c(\beta,m)\in\Z~.\]

By the main theorem of \cite[Theorem A]{howardmadapusi}, there exists a Borcherds products $\psi(f)$  associated to $f$ which defines, after multiplying $f$ by a suitable integer, a rational section of $\boldsymbol{\omega}^{\frac{c(0,0)}{2}}$ over $\Q$ and its divisor in $\mathcal{M}$ is equal to: 

\[\mathrm{div}(\psi(f))=\sum_{(\beta,m)}c(\beta,-m)\mathcal{Z}(\beta,m).\]

Let $\Sigma$ an admissible polyhedral cone decomposition and let $\mathcal{Z}^{tor}(\beta,m)$ be the completed divisor as defined in \Cref{completed-divisor}. Then by \cite[Proof of Thm 3.1]{tayouboundary}, the divisor of the Borcherds products on $\mathcal{M}^\Sigma$ is equal to: 
\[\mathrm{div}(\psi(f))=\sum_{(\beta,m)}c(\beta,-m)\mathcal{Z}^{tor}(\beta,m)~.
\]
In fact, the above relation can be upgraded into an equality by \cite[Equation (1.2.2)]{howardmadapusi} in $\widehat{\CH}^1(\cM^{\Sigma},\mathcal{D}_{pre})_{\Q}$:
\[\widehat{\mathrm{div}}(\psi(f))=\sum_{(\beta,m)}c(\beta,-m)\widehat{\mathcal{Z}}^{tor}(\beta,m)~.\]
On the other hand, 
\[\widehat{\mathrm{div}}(\psi(f))=\frac{c(0,0)}{2}\widehat{\boldsymbol{\omega}}~.\]
Hence we get the following equality: 
\[\frac{c(0,0)}{2}\widehat{\boldsymbol{\omega}}=\sum_{(\beta,m)}c(\beta,m)\widehat{\mathcal{Z}}^{tor}(\beta,m)~.
\]


\subsection{Expansions at the cusp}
We assume now that we are given two quadratic lattices $(rL,Q)\subset (L,Q)$ where $L$ is a maximal lattice, almost self-dual at $r$. Let $\Sigma$ be a $K_r$-admissible polyhedral cone decomposition and let $\widetilde{\mathcal{M}}^\Sigma_r$ be the toroidal compactification of the integral model of the GSpin Shimura variety associated to $(rL,Q)$ constructed in \Cref{level-structure}. It is normal, proper and flat over $\Z$.
\medskip 

We recall in this section the theory of integral $q$-expansions at the cusps following \cite[\S\S 5,8]{howardmadapusi}. We assume that $(rL,Q)$ is isotropic and let $(\Xi,\sigma)$ be a toroidal stratum representative such that $\sigma$ is top dimensional. 

We will first describe the Fourier-Jacobi expansion over $\C$ and then over $\Z_{(p)}$ where $p$ is a prime number. Associated to the cusp label representative $\Xi$, there is an admissible parabolic subgroup $P_\Xi\subset G$, a connected component $\cD^{\circ}$ of $\cD$ and an element $h\in G(\mathbb{A}_f)$. Such cusp label representative determines a mixed Shimura datum $(Q_\Xi,\cD_\Xi)$, see \cite[\S 4.4]{howardmadapusi}. The unipotent radical $W_\Xi$ and its center $U_\Xi$ are both equal and are described at the level of $\Q$-points by: 
\[U_\Xi(\Q)\simeq K_\Q\otimes I_\Q\]
where $I_\Q$ is the $\Q$-isotropic line determined by the cusp label representative $\Xi$, $I=rL\cap I_\Q$, and $K=I^{\bot}/I$. Define the $\Z$-lattice $\Gamma_\Xi=K_\Xi\cap U_\Xi(\Q)$ and the torus: 
\[T_\Xi=\Gamma_\Xi(-1)\otimes \mathbb{G}_m.\]

The level $K_r$ determines a mixed Shimura variety $M_{\Xi}$ associated to the mixed Shimura datum $(Q_\Xi,\cD_\Xi)$. Let $K_{\Xi0}$ be the compact open subgroup of $Q_\Xi(\mathbb{A}_f)$ determined as in the end of page 220 of \cite{howardmadapusi}. It defines another mixed Shimura variety $M_{\Xi0}$ over $\Q$ associated to the same datum $(Q_\Xi,\cD_\Xi)$ and an \'etale morphism of Deligne--Mumford stacks $M_{\Xi0}\rightarrow M_{\Xi}$.

The toroidal stratum representative $(\Xi,\sigma)$ determines partial compactifications $M_{\Xi}(\sigma)$, $M_{\Xi0}(\sigma)$ and  $0$-dimensional boundary component $Z^{\Xi,\sigma}$ of $M_r^\Sigma$, $M_\Xi(\sigma)$, and $M_{\Xi 0}(\sigma)$. We denote by $\widehat{M}_{r}^\Sigma$, resp.  $\widehat{M}_{\Xi}(\sigma)$, $\widehat{M}_{\Xi 0}(\sigma)$, the formal completion of $M_r^\Sigma$, resp. $M_{\Xi}(\sigma)$, $M_{\Xi 0}(\sigma)$ along $Z^{\Xi,\sigma}$. By a theorem of Pink \cite[Corollary 7.17, Theorem 12.4]{pink}, see also \cite[\S 2.6]{howardmadapusi} which is our reference, we have an isomorphism 
\[\widehat{M}_{r}^\Sigma\simeq \widehat{M}_{\Xi}(\sigma)~.\]

Then by \cite[Proposition 4.6.2]{howardmadapusi}, there exists $K_0\subset \mathbb{A}_f^{\times}$ compact open subgroup such that we have the following commutative diagram of formal Deligne--Mumford stacks over $\C$: 

\begin{align}\label{diagram-complex}
\xymatrix{
\bigsqcup_{a\in\Q^{\times}_{>0}\backslash \mathbb{A}_f^{\times}/K_0} \widehat{{T}}_{\Xi}(\sigma)_{/\C} \ar[rr]^{\simeq}\ar[d] && \widehat{{M}}_{\Xi,0}(\sigma)/\C\ar[d]\\ 
 \widehat{M}^{\Sigma}_{r/\C} \ar[rr]^{\simeq} && \widehat{{M}}_{\Xi}(\sigma)/\C, 
}
\end{align}
such that the vertical arrows are formally \'etale surjections and 
\[\widehat{{T}}_{\Xi}(\sigma)=\mathrm{Spf}\left(\Q[[q_\alpha]]_{\underset{(\alpha\cdot\sigma)\geq 0}{\alpha\in\Gamma_\Xi^\vee(1)}}\right)~.\]

Now given a section $\psi$ of $\boldsymbol{\omega}_r^{\otimes k}$, we get by \cite[Equation (4.6.10)]{howardmadapusi} a trivialization, \emph{the Fourier-Jacobi expansion} on each copy of $\widehat{{T}}_{\Xi}(\sigma)_{/\C}$ indexed by $a\in \Q^{\times}_{>0}\backslash \mathbb{A}_f^{\times}/K_0$: 
\[\mathrm{FJ}^{(a)}(\psi)=\sum_{\underset{(\alpha\cdot\sigma)\geq 0}{\alpha\in\Gamma_{\Xi}^\vee(1)}}\mathrm{FJ}_\alpha^{(a)}(\psi)\cdot q_\alpha\in\C[[q_\alpha]]_{\underset{(\alpha\cdot\sigma)\geq 0}{\alpha\in\Gamma_{\Xi}^\vee(1)}}~.\]
\medskip

Let $f$ be a weakly holomorphic modular from of weight $1-\frac{n}{2}$ with respect to $\overline{\rho}_{rL}$ and with integral principal part. Let $\psi(f)$ be the associated Howard-Madapusi-Borcherds product, which is a rational section of $\boldsymbol{\omega}_r^{\frac{c(0,0)}{2}}$. 
Let $F$ be the abelian extension of $\Q$ determined by the reciprocity isomorphism in class field theory: 
\[\mathrm{rec}:\Q^{\times}_{>0}\backslash \mathbb{A}_f^{\times}/K_0\simeq \mathrm{Gal}(F/\Q)~.\]

By \cite[Proposition 5.4.2]{howardmadapusi}, for every $a\in \Q^{\times}_{>0}\backslash \mathbb{A}_f^{\times}/K_0$, the Borcherds product $\psi(f)$ has a Fourier-Jacobi expansion given as follows: 
\begin{align}\label{bp-expansion}
    \mathrm{FJ}^{(a)}(\psi(f))=\kappa ^{(a)}A^{\rec(a)}q_{\alpha(\rho)}\cdot\mathrm{BP}(f)^{\rec(a)}~,
\end{align}
where $\kappa^{(a)}\in \C$ is a constant of absolute value $1$, and \[\mathrm{BP}(f)\in\mathcal{O}_F[[q_\alpha]]_{\underset{(\alpha\cdot\sigma)\geq 0}{\alpha\in\Gamma^{\vee}_\Xi(1)}}~,\]
is the infinite product: 
\[\mathrm{BP}(f)=\prod_{\underset{(\lambda\cdot\mathscr{W})>0}{\lambda\in (I^\bot/I)^\vee}}\prod_{\mu\in hL^{\vee}/hL}\left(1-\zeta_{\mu}\cdot q_{\alpha(\lambda)}\right)^{c(h^{-1}\mu,-Q(\lambda))}.\]
In the product above, $\mathscr{W}$ is a Weyl chamber as defined in \cite[Equation (5.3.1)]{howardmadapusi} such the interior of the cone $\sigma$ is isomorphic to an open subset of $\mathscr{W}$. The number $\zeta_{\mu}$ is a root of unity of order dividing $|\frac{1}{r}L^\vee/rL|$. 

Finally, the constant $A$ is given as follows, see \cite[Equation (5.3.6)]{howardmadapusi}: let $I_\Q$ be be as before the isotropic line corresponding to the cusp label representative $\Xi=(\cD,P,h)$ and let $\ell$ be a generator of $I\cap h\cdot rL$. Let $N$ be the order $\ell$ in $(h\cdot rL_\Z)^\vee/(h\cdot rL_\Z)$. 
Then 
\[A=\prod_{\underset{x\neq 0}{x\in \Z/N\Z}}\left(1-e^{\frac{2\pi i x}{N}}\right)^{c(\frac{xh^{-1}\ell}{N},0)}.\]
\begin{remarque}
There is a $(2i\pi)^{\frac{c(0,0)}{2}}$ factor in \cite[Proposition 5.4.2]{howardmadapusi} that disappeared from \Cref{bp-expansion} and the reason is that Howard-Madapusi already rescaled the original Borcherds constructed by Borcherds in \cite{borcherds-inventiones-automorphic}, denoted by $\Psi(f)$ in {\it loc. cit.}, by the factor $(2i \pi)^{\frac{c(0,0)}{2}}$ to obtain $\psi(f)$. 
\end{remarque}

\begin{lemma}
The constant $\kappa^{(a)}$ is equal to $1$.

\end{lemma}
\begin{proof}
This is one of the main difficulties that was overcome in \cite{howardmadapusi} in the course of the construction of $\psi(f)$. When the lattice $(L,Q)$ satisfies the assumptions of \cite[Proposition 9.1.2]{howardmadapusi}, then it follows from the construction in the discussion of the middle of the page 283\footnote{Page 97 in arXiv version 2.} \textit{loc. cit.} that we can take $\kappa^{(a)}=1$. In general, $\psi(f)$ is defined as the quotient of two regularized Borcherds products, see \cite[Equation (9.2.8)]{howardmadapusi}, associated to lattices $L_1=L\oplus \Lambda_1$ and $L_2=L\oplus \Lambda_2$ where $\Lambda_1$ and $\Lambda_2$ are certain self-dual lattices of signature $(24,0)$, see \S 9.2. The lattices $L_1$ and $L_2$ satisfy the above conditions and hence the constant $\kappa^{(a)}$ appearing in their Borcherds product are equal to one. Moreover, the regularization involve certain ``analytic obstruction terms'' defined in Section 6.5. {\it loc. cit. } which are simply the equations of the special divisors in the universal cover $\cD$ and whose Fourier expansions do not contribute to $\kappa^{(a)}$. Since the different Fourier expansions of the Borcherds products are compatible with each other, we conclude that $\kappa^{(a)}=1$ in our case too.

\end{proof}

\subsection{Integral theory}

Let $p$ be a prime number. We now extend the results from the previous section to $\Z_{(p)}$. We still assume that $rL$ has an isotropic vector, which is always true if $n\geq 3$. Let \[ \widehat{\mathcal{T}}_{\Xi}(\sigma)=\mathrm{Spf}\left(\Z_{(p)}[[q_\alpha]]_{\underset{(\alpha\cdot\sigma)\geq 0}{\alpha\in\Gamma_\Xi^\vee(1)}}\right)~,\]

and let $R$ be the localization of $\mathcal{O}_F$ at a prime $\mathfrak{P}\subset\mathcal{O}_F$ above $p$.
\begin{proposition}\label{charts}
There is a unique morphism \[\bigsqcup_{a\in \Q^{\times}_{>0}\backslash \mathbb{A}_f^{\times}/K_0} \widehat{\mathcal{T}}_{\Xi}(\sigma)_{/R}\rightarrow \widehat{\widetilde{\mathcal{M}}_r^\Sigma}~.\]
of formal Deligne--Mumford stacks which agrees with \Cref{diagram-complex} by base change to $\C$, and such that for any $s$ in the source with image $t$, the induced map on \'etale local rings is faithfully flat.    
\end{proposition}
For primes $p$ where $rL$ is maximal, i.e., those primes who do not divide $r$, the result above is \cite[Proposition 8.2.3]{howardmadapusi}. The proof of Proposition 8.2.3 in \cite{howardmadapusi} uses Proposition 8.1.1 as the main input. The latter relies on Theorem 4.1.5 from \cite{madapusitor} and in fact the maximality assumption is not needed in the latter result.



\begin{lemma}\label{flat-Borcherds}
Assume that $A=1$. Then the divisor of the Borcherds product $\psi(f)$ in $\widetilde{\mathcal{M}}_{r,(p)}$ is flat over $\Z_{(p)}$.
\end{lemma}
\begin{proof}
   As $\widetilde{\mathcal{M}}_{r,(p)}$ is flat and normal over $\Z_{(p)}$, it is enough to show by \Cref{normality-flatness} that the divisor of $\psi(f)$ in $\widetilde{\mathcal{M}}_{r,(p)}$ does not contain any irreducible component of the special fiber of $\widetilde{\mathcal{M}}_{r,\F_p}$. Since the Fourier-Jacobi expansion of $\psi(f)$ in \Cref{bp-expansion} is not zero modulo $\mathfrak{P}$, we deduce using the faithful flatness of \Cref{charts}, that the divisor of $\psi(f)$ in every irreducible component that meets the cusp is flat. Hence it is enough to prove that each irreducible component of the special fiber $\widetilde{\mathcal{M}}_{r,\F_p}$ meets the zero cusp of $\mathcal{M}_r^{\Sigma}$. 
   
   Recall that we have a finite morphism \[\eta:\widetilde{\mathcal{M}}_{r,(p)}\rightarrow \mathcal{M}_{(p)}\] 
   and $\mathcal{M}_{(p)}$ is smooth over $\Z_{(p)}$, in particular, every irreducible component of $\mathcal{M}_{\F_p}$ is connected and meets the zero cusp. Since the morphism $\eta$ is finite, every irreducible component of $\widetilde{\mathcal{M}}_{r,\F_p}$ maps surjectivity to an irreducible component of $\mathcal{M}_{\F_p}$. Hence every irreducible component of $\widetilde{\mathcal{M}}_{r,\F_p}$ meets the $0$-cusp of $\widetilde{\mathcal{M}}_{r,\F_p}^\Sigma$, which concludes the proof. 
\end{proof}

\subsection{Construction of a flat Borcherds product}
We will construct in this section Borcherds products which satisfy the conditions of \Cref{flat-Borcherds}.
\medskip

Let $\beta\in \frac{1}{r}L^{\vee}/rL$. For every $m\in Q(\beta)+\Z$, let $a_{\beta,m}$ be the linear form on the space of cusp forms $S_{1+\frac{n}{2}}(\rho_{rL})$ which maps a cusp from $g$ to its $(\beta,m)^{th}$-Fourier coefficient. Then there exists a finite set of indices $I$ such that $a_{\beta,m_i}$ generates the $\Q$-vector space  
\[\mathrm{Span}(a_{\beta,m}|\,m\in Q(\beta)+\Z)\subset {S_{1+\frac{n}{2}}(\rho_{rL})}^{*}~,\]
where ${S_{1+\frac{n}{2}}(\rho_{\rho_{rL}})}^{*}$ is the dual of the space of cusp forms. 

Let $(g_i)_{i\in I}$ be a dual family \footnote{We don't require $a_{\beta,m_i}(g_i)=1$, but only $a_{\beta,m_i}(g_j)=0$ for $j\neq i$.} of cusp forms to the family $(a_{\beta,m_i})_{i\in I}$ and we can assume that the $(g_i)$ have integral Fourier coefficients by \cite{mcgraw}. Then for each  $m$, there exists $c_{i}(\beta,m)\in \Q $ such that we can write 
\[a(\beta,m)=\sum_{i\in I} c_i(\beta,m)a(\beta,m_i)\]
and $g_i(\beta,m)=c_i(\beta,m)g_i(\beta,m_i)$. Standard estimates on growths of coefficients of cups forms show that: 
\begin{align}\label{eq:estimate-coeff-2}
    c_i(\beta,m)=O(c(\beta,m))~,
\end{align}
where $c(\beta,m)$ is the $(\beta,m)$-coefficient of the Eisenstein series introduced in the paragraph before \Cref{global-bound}.

By \cite[Theorem 1.17]{bruinier}, there exists a weakly holomorphic modular form $\tilde{f}_m\in M_{1-\frac{n}{2}}^{!}(\overline{\rho}_{rL})$ such that its principal part is equal to 
\[(v_\beta+v_{-\beta})q^{-m}-\sum_{i\in I}c_i(\beta,-m)q^{-m_i}(v_\beta+v_{-\beta})~.\]

Let $d=|\frac{1}{r}L^\vee/rL|-1$ and let 
\[C(m)=\left((c(\gamma,0)\right)_{\underset{\gamma\neq0}{\gamma\in \frac{1}{r}L^\vee/rL}}\in \Q^{d}\]
be the vector of constant Fourier coefficients of $\tilde{f}_m$. Then the span of $(C(m))_{m\geq 1}$ is a finite dimensional vector space of $\Q^{d}$ which admits a basis given by $(C(m_j))_{j\in J}$ for some finite set $J$. 

Finally for any $m$, there exists coefficients $u_j(m)\in \Q$ such that  $u_j(m)=O(c(\beta,m))$ and
\[C(m)=\sum_{j\in J} u_j(m)C(m_j).\]
Define
\[f_m=\tilde{f}_m-\sum_{j\in J} u_j(m)\tilde{f}_{m_j}~.\]
Then by construction, all the $(\gamma,0)^{th}$-Fourier coefficients of $f_m$ vanish, except possibly the $(0,0)^{th}$-coefficient. Moreover, its principal part is equal to: 

\begin{multline*}
    (v_\beta+v_{\beta})q^{-m}-\sum_{i\in I}c_i(\beta,m)(v_\beta+v_{-\beta})q^{-m_i}\\-\sum_{j\in J}u_j(m)\left((v_\beta+v_{-\beta})q^{-m_j}-\sum_{i\in I}c_i(\beta,m_j)(v_\beta+v_{-\beta})q^{-m_i}\right)~.
\end{multline*}
The latter can be rewritten as: 
\[(v_\beta+v_{-\beta})q^{-m}+\sum_{\ell\in \tilde{I}}z_\ell(m)(v_\beta+v_{-\beta})q^{-m_\ell}~,\]
where $\tilde{I}$ is finite set independent of $m$, $m_\ell$ are independent of $m$, and $z_\ell(m)$ are rational numbers that satisfy: 
\begin{align}\label{eq:estimate-coeff-1}
    z_\ell(m)=O(c(\beta,m))~.
\end{align}
\medskip 

Up to multiplying $f_m$ by an integer, let $\psi(f)$ be the Borcherds product associated to $f$ as in \Cref{background}. Then $\psi(f_m)$ is a section of $\boldsymbol{\omega}_r^{\frac{c(0,0)}{2}}$ and we have the following relation in $\widehat{\CH}^1(\cM^{\Sigma}_1,\mathcal{D}_{pre})_{\Q}$ 
\begin{align}\label{green-relation}
    \frac{c(0,0)}{2}\widehat{\boldsymbol{\omega}}_r=\widehat{\mathrm{div}}(\psi(f_m))=\widehat{\mathcal{Z}}^{tor}(\beta,m)+\sum_{\ell\in \tilde{I}} {z_\ell(m)}\widehat{\mathcal{Z}}^{tor}(\beta,m_\ell).
\end{align}

Assume now that $\beta\in L/rL$ and we will extend the above relation above the primes $p\in \Omega$. Let $\widetilde{\mathcal{M}}_r^\Sigma$ be the proper, flat integral model over $\Z$ extending $\mathcal{M}_r^\Sigma$.

Assume also that $rL$ has an isotropic vector. Then, by choice of the Borcherds product $\psi(f_m)$, the constant $A$ from \Cref{bp-expansion} is equal to $1$, hence by lemma \Cref{flat-Borcherds}, the divisor of the Borcherds $\psi(f_m)$ is flat over $\Z_{(p)}$. By \Cref{fltaness-special-cycles}, the special divisors $\mathcal{Z}(\beta,m)$ are flat over $\Z_{(p)}$ and the boundary divisors are flat over $\Z_{(p)}$ by \cite[Theorem 1]{madapusitor}. Hence \Cref{green-relation} holds over $\Z_{(p)}$ for all $p\in\Omega$. Hence we conclude. 


\begin{proposition}\label{prop:global-relation}
    Let $\beta\in L/rL$ and assume that $rL$ is isotropic. Then we have in $\widehat{\CH}^1(\widetilde{\cM_r}^{\Sigma},\mathcal{D}_{pre})_{\Q}$ : 
    \[\frac{c(0,0)}{2}\widehat{\boldsymbol{\omega}}_r=\widehat{\mathrm{div}}(\psi(f_m))=\widehat{\mathcal{Z}}^{tor}(\beta,m)+\sum_{\ell\in \tilde I} {z_\ell(m)}\widehat{\mathcal{Z}}^{tor}(\beta,m_\ell),\]
    where the integers $m_\ell$ are independent of $m$.
\end{proposition}

\subsection{Summary}

Let $m\in \N$, then by \Cref{prop:global-relation} provides a section $\psi(f_m)$ of $\boldsymbol{\omega}_r^{\frac{c(0,0)}{2}}$ that satisfies the following relation in $\widehat{\CH}^1(\widetilde{\cM_r^{\Sigma}},\mathcal{D}_{pre})_{\Q}$: 
\[\widehat{\mathrm{div}}(\psi(f_m))=\widehat{\mathcal{Z}}^{tor}(\beta,m)+\sum_{\ell\in \tilde{I}}z_\ell(m)\widehat{\mathcal{Z}}^{tor}(\beta,m_\ell).\]
From which it follows that:  
\[h_{\widehat{\mathcal{{Z}}}^{tor}(\beta,m)}(\mathscr{S})=\frac{c(0,0)}{2}h_{\widehat{\boldsymbol{\omega}}_r}(\mathscr{S})-\sum_{\ell\in\tilde{I}} z_\ell(m)\cdot h_{\widehat{\mathcal{{Z}}}^{tor}(\beta,m_i)}(\mathscr{S})~.\]

Using \Cref{eq:estimate-coeff-1} and \Cref{eq:estimate-coeff-2}, we get 
\begin{align}\label{eq:estimate-height}
    h_{\widehat{\mathcal{{Z}}}(\beta,m)}(\mathscr{S})=O(c(\beta,m))~.
\end{align}
Using \Cref{completed-divisor}, we have the equality: 
\begin{align*}
\left(\mathscr S.{\mathcal{{Z}}}^{tor}(\beta,m)\right)&=\left(\mathscr S.\mathcal{{Z}}(\beta,m)\right)+\sum_{\Upsilon}\mu_\Upsilon(\beta,m)\left(\mathscr S.\cB^\Upsilon\right)\\
&+\sum_{(\Xi,\sigma)}\mu_\Xi(\beta,m)\left(\mathscr S.\cB^{\Xi,\Sigma}\right)\numberthis \label{eq:divisor-toroidal}~,
\end{align*}
and from \cite[Proposition 4.13]{tayouboundary}, we have the following estimate as $m\rightarrow\infty$: 
\begin{enumerate}
    \item For any type II toroidal stratum representative $\Upsilon$, we have \begin{align}\label{eq:estimate-type-II}
        \mu_\Upsilon(m)\ll_\epsilon m^{\frac{b}{2}-1+\epsilon}~.
    \end{align}
    \item For any type III toroidal stratum representative $(\Xi,\sigma)$ such that $\sigma$ is a ray, we have  
    \begin{align}\label{eq:estimate-type-III}
        \mu_{\Xi,\sigma}(m)\ll_\epsilon m^{\frac{b-1}{2}+\epsilon}.
    \end{align}
\end{enumerate}

To prove \Cref{global-bound}, we first  use \Cref{eq:divisor-toroidal} to write: 
\begin{align*}
\sum_{\mathfrak{P}\subset \mathcal{O}_K}(\mathcal{Z}&(\beta,m).\mathscr{S})_\mathfrak{P}\log|\mathcal{O}_K/\mathfrak{P}|+\sum_{x\in\mathscr{S}(\C)}\Phi_{\beta,m}(x)= h_{\widehat{\mathcal{{Z}}}(\beta,m)}(\mathscr{S})\\&-\sum_{\Upsilon}\mu_\Upsilon(\beta,m)\left(\mathscr S.\cB^\Upsilon\right)-\sum_{(\Xi,\sigma)}\mu_\Xi(\beta,m)\left(\mathscr S.\cB^{\Xi,\Sigma}\right)~.
\end{align*}
Then we can bound the right hand side above using the estimates from \Cref{eq:estimate-height}, \Cref{eq:estimate-type-II}, and \Cref{eq:estimate-type-III}. This finishes the proof.

\section{Archimedean estimates}
Our goal in this section is to prove \Cref{local-infinite}. We follow the approach explained in \cite[\S 5]{sstt}. 

\subsection{Development of the Green function}
The Green function $\Phi_{\beta,m}$ has an explicit expression due to Bruinier \cite[\S 2]{bruinier} and which we recall following \cite[\S 5]{sstt}. 
\medskip 

Let $(L,Q)$ be a quadratic lattice of signature $(n,2)$. Let $k=1+\frac{n}{2}$, $\beta\in L^\vee/L$ and $s>\frac{k}{2}$ a real number. Let:
\begin{align*}
F(s,z)=H\left(s-1+\frac{k}{2},s+1-\frac{k}{2},2s;z\right), 
\end{align*}
where
\begin{align*}
    H(a,b,c;z)=\sum_{n\geq 0}\frac{(a)_n(b)_n}{(c)_n}\frac{z^n}{n!}
\end{align*}
is the Gauss hypergeometric function as in \cite[Chapter 15]{handbook}, and $(a)_n=\frac{\Gamma(a+n)}{\Gamma(a)}$ for $a,b,c,z\in \C$ and $|z|<1$.

For $x\in\cD$, define: 
\begin{multline}\label{green-function}
\phi_{\beta,m}(x,s)=2\frac{\Gamma(s-1+\frac{k}{2})}{\Gamma(2s)}\\ \sum_{Q(\lambda)=m, \lambda\in \beta+L}\left(\frac{m}{m-Q(\lambda_x)}\right)^{s-1+\frac{k}{2}} F\left(s,\frac{m}{m-Q(\lambda_x)}\right)~.
\end{multline}

Then $\phi_{\beta,m}(x,s)$ admits a meromorphic continuation to the complex plane with a pole at $s=\frac{k}{2}$ with residue $-c(\beta,m)$. 
We define then: 
\begin{align}\label{limit}
\phi_{\beta,m}(x)=\lim_{s\rightarrow \frac{k}{2}}\left(\phi_{\beta,m}(x,s)+\frac{c(\beta,m)}{s-\frac{k}{2}}\right).
\end{align}

Let $s\mapsto C(\beta,m,s)$ be the holomorphic function for $\mathrm{Re}(s)>1$ defined in \cite[Equation (3.3)]{sstt}, see also \cite[Equation (3.22)]{bruinierintegrals}. Then define: 
 
\begin{align}\label{equ:b}
b(\beta,m,s)=-\frac{C\left(\beta,m,s-\frac{k}{2}\right).\left(s-1+\frac{k}{2}\right)}{\left(2s-1\right).\Gamma\left(s+1-\frac{k}{2}\right)}.
\end{align} 

By \cite[Proposition 2.11]{bruinier}, we can write for $x\in\cD$: 
\[\Phi_{\beta,m}(x)=\phi_{\beta,m)}(x)-b'(\beta,m,\frac{k}{2}).\]

\subsection{Estimate on $b'(\beta,m,\frac{k}{2})$}
We make the following assumptions in this section: $L$ is maximal and $\beta\in L/rL$ has torsion coprime to the discriminant of $L$. 
Our goal in this section to prove the following theorem. 
\begin{theorem}
Let $D\geq 1$ be an integer. For $m\rightarrow \infty$ representable by $\beta+rL$ and such that $\sqrt{\frac{m}{D}}\in\Z$, we have: 
\[b'(\beta,m,\frac{k}{2})=|c(\beta,m)|\log(m)+o\left(c(\beta,m)\log(m)\right).\]
\end{theorem}

\begin{proof}
The theorem above has been proved in \cite[Proposition 5.2]{sstt} under the assumption that $L$ is maximal and $\beta=0$. We recall the main steps here and make the appropriate modifications.

Taking logarithmic derivatives at $s=\frac{k}{2}$ in \Cref{equ:b} yields: 
\begin{align*}
\frac{b'(\beta,m,\frac{k}{2})}{b(\beta,m,\frac{k}{2})}=\frac{C'(\beta,m,0)}{C(\beta,m,0)}-\frac{2}{b}-\Gamma'(1)~.
\end{align*}
Let  \[N_{\beta,m}(a)=|\{\lambda\in L/aL|\, \lambda=\beta \pmod{rL},\, Q(\lambda)=m \pmod{a}\}|.\]
Let also $d_\beta$ denote the order of $\beta$ in $\frac{1}{r}L^\vee/rL$ and for a prime number $p$, let: 
\[w_p=1+2v_p(2md_\beta)~.\]

Define the polynomial $\mathrm{L}_{\beta,m}^{(p)}(t)$:  
\[\mathrm{L}_m^{(p)}(t)=N_{\beta,m}(p^{w_p})t^{w_p}+(1-p^{r-1}t)\sum_{n=0}^{w_p-1}N_{\beta,m}(p^n)t^n\in\Z[t]~.\]

For $s\in\C$, define the function $\sigma_{\beta,m}(s)$:   
\begin{align}\label{sigma}
\sigma_{\beta,m}(s)=\left\{\begin{array}{ll}
\underset{p\backslash 2d_{\beta}^2m\det(L)}{\prod}\frac{\mathrm L^{(p)}_{\beta,m}\left(p^{1-\frac{r}{2}-s}\right)}{1-\chi_{D_{0}}(p)p^{-s}},\quad \textrm{if}\quad  r=2+n\quad \textrm{is even},\\
\underset{p\backslash 2d_{\beta}^2m\det(L)}{\prod}\frac{1-\chi_{D_{0}}(p)p^{\frac{1}{2}-s}}{1-p^{1-2s}}\cdot\mathrm L^{(p)}_{\beta,m}\left(p^{1-\frac{r}{2}-s}\right),\quad \textrm{if}\quad r\quad \textrm{is odd}.
\end{array}
\right.
\end{align}
Here, $\chi_{D_{0}}$ is the quadratic character associated to a fundamental discriminant $D_{0}$ of the number field $\Q(\sqrt{D})$ where $D$ is defined by
\begin{align*}
&(-1)^{\frac{r}{2}}\det(L),\, \textrm{if}\, r\,\textrm{is even}.\\
&2(-1)^{\frac{r+1}{2}}d_{\beta}^2m\det(L),\, \textrm{otherwise}.
\end{align*}

By our choice of $m$, the fundamental discriminant is independent of $m$, hence \cite[Theorem 4.11, (4.73), (4.74)]{bruinierintegrals} implies: 
\begin{align*}
\frac{C'(\beta,m,0)}{C(\beta,m,0)}=\log(m)+\frac{\sigma_{\beta,m}'(k)}{\sigma_{\beta,m}(k)}+O(1).
\end{align*}

It suffices thus  to show that $\frac{\sigma_{\beta,m}'(k)}{\sigma_{\beta,m}(k)}=o(\log(m)$. Taking the logarithmic derivative in (\ref{sigma}) at $s=k$, we get for $r$ even
\begin{align*}
\frac{\sigma_{\beta,m}'(k)}{\sigma_{\beta,m}(k)}=-\sum_{p\backslash 2d_{\beta}^2m\det(L)}\left(\frac{p^{1-r}\mathrm L_{\beta,m}^{(p)'}\left(p^{1-r}\right)}{\mathrm L_{\beta,m}^{(p)}\left(p^{1-r}\right)}+\frac{\chi_{D_0}(p)}{p^k-\chi_{D_0}(p)}\right)\log(p)~,
\end{align*} 
and for $r$ odd 
\begin{align*}
\frac{\sigma_{\beta,m}'(k)}{\sigma_{\beta,m}(k)}=-\sum_{p\backslash 2d_{\beta}^2m\det(L)}\left(\frac{p^{1-r}\mathrm L_{\beta,m}^{(p)'}\left(p^{1-r}\right)}{\mathrm L_{\beta,m}^{(p)}\left(p^{1-r}\right)}-\frac{\chi_{D_0}(p)}{p^{k-\frac{1}{2}}-\chi_{D_0}(p)}+\frac{2}{p^{2k-1}-1}\right) \log(p)
\end{align*}

We have $\mathrm L_{\beta,m}^{(p)}(p^{1-r})=N_{\beta,m}(p^{w_p})p^{(1-r)w_p}$ and
\[\mathrm{L}^{(p)'}_{\beta,m}(p^{1-r})=w_pN_{\beta,m}(p^{w_p})p^{(1-r)(w_p-1)}-\sum_{n=0}^{w_p-1}N_{\beta,m}(p^n)p^{(n-1)(1-r)}~.\]
Hence 
\begin{multline*}
\left|\frac{p^{1-r}\mathrm{L}_{\beta,m}^{(p)'}\left(p^{1-r}\right)}{\mathrm{L}_{\beta,m}^{(p)}\left(p^{1-r}\right)}\right|=\left|w_p-\sum_{v=0}^{w_p-1}\frac{N_{\beta,m}(p^v)}{N_{\beta,m}(p^{w_p})}p^{(v-w_p)(1-r)}\right|\\=\left|w_p-\sum_{v=0}^{w_p-1}\frac{\mu_{p}(\beta,m,v)}{\mu_{p}(\beta,m,w_p)}\right|,
\end{multline*}
where $\mu_{p}(\beta,m,v)=p^{-n(r-1)}N_{\beta,m}(p^v)$. The proof of \cite[Proposition 5.2]{sstt} shows then it is enough to prove \Cref{lemma:bound} below. Assuming this lemma, we get: 
\begin{align*}
\left|\sum_{p\backslash 2d_{\beta}^2m\det(L) }\frac{p^{1-r}\mathrm L_{\beta,m}^{(p)'}\left(p^{1-r}\right)}{\mathrm L_{\beta,m}^{(p)}\left(p^{1-r}\right)}\cdot\log(p)\right|&\leq C\sum_{p\backslash 2d_{\beta}^2m\det(L)}\frac{\log(p)}{p}\\&=O(\log\log(m))
\end{align*}
\end{proof}



\begin{lemma}\label{lemma:bound} 
Let $\beta\in L/rL$ be a primitive element, $m\in \Z$ representable by $\beta+rL$, and $p$ a prime number. Then there exists a constant $C$ independent of $m$ and $p$ such that  
\begin{align*}
\left|\omega_p-\sum_{v=0}^{w_p-1}\frac{\mu_p(\beta,m,v)}{\mu_p(\beta,m,w_p)}\right|\leq\frac{C}{p} 
\end{align*}
\end{lemma}
\begin{proof}
For $p$ coprime to $r$, the result follows from \cite[Proposition 4.1]{sstt}. Hence we can assume that $p$ divides $r$. By assumption, $L$ is unimodular at $p$ and $\beta$ is primitive $r$-torsion, hence every solution $x\in \beta+r(L/p^kL)$ is good in the sense of \cite[Definition 3.1]{hanke}. Let $\delta=1+v_p(r)$
Then using that $\beta\neq 0$ modulo $p$, we get for $v\geq\delta$: 
\[N_{\beta,m}(p^v)=p^{(v-\delta)(r-1)}\cdot N_{\beta,m}(p^\delta)=p^{(v-\delta+1)(r-1)},\]
and $\mu_{p}(\beta,m,v)=\mu_{p}(\beta,m,\delta)=\frac{1}{p}$. As for $v<\delta$, we have $N_{\beta,m}(p^v)=1$. 
Hence: 
\begin{align*}
\left|\omega_p-\sum_{v=0}^{w_p-1}\frac{\mu_p(\beta,m,v)}{\mu_p(\beta,m,w_p)}\right|=\frac{1}{\mu_p(\beta,m,\delta)}\left|\sum_{v=0}^{
\delta-1}\frac{1}{p}-\frac{1}{p^{v(r-1)}}\right|\ll \frac{1}{p}~.
\end{align*}

\end{proof}

\begin{lemma}\label{representation}
    Let $L$ be a maximal lattice and let $r\geq 1$ coprime to the discriminant of $L$. Let $\beta\in L/rL$ be a primitive $r$ torsion element. Then any integer $m$ large enough such that $m\equiv Q(\beta)\pmod{r}$, is representable by $\beta+rL$.
\end{lemma}
\begin{proof}

There is no local obstruction to finding $m$ and hence the argument of \cite[Corollary 4.7]{sstt} still applies. 
\end{proof}

\subsection{Proof of \Cref{local-infinite}}
We assume in this section that $L$ is maximal lattice, self-dual at primes dividing $r$. Let $\beta \in L/rL$, $m\in Q(\beta)+r\Z$ and let $\phi_{\beta,m}$ be the Green function defined by \Cref{green-function}. 

For $x\in \cD$, define 
\[A(\beta,m,x)=-2\sum_{\underset{|Q(\lambda_x)|\leq 1,Q(\lambda)=1}{\sqrt{m}\lambda\in \beta+rL}}\log(|Q(\lambda_x)|).\]
Then by \cite[Proposition 5.4]{sstt}, we have: 

\[\phi_{\beta,m}(v)=A(\beta,m,x)+O(m^{\frac{b}{2}})~.\]
The reference only proves it for $\beta=0$ and $L$ maximal, but the same proof applies with minor changes. 
\begin{proposition}\label{thAbound}
There exists a subset $S_{\bad} \subset \Z_{>0}$ of logarithmic asymptotic density zero such that for every $m\notin S_{\bad}$, we have
\begin{equation*}
    A(\beta,m,x)=o(m^{\frac{b}{2}}\log(m)).
\end{equation*}
\end{proposition}
\begin{proof}
Let \[A(m,x)=-2\sum_{\underset{|Q(\lambda_x)|\leq 1,Q(\lambda)=1}{\sqrt{m}\lambda\in L}}\log(|Q(\lambda_x)|)~.\]
    Notice that $A(\beta,m,x)\geq 0$ and that \[\sum_{\beta\in L/rL}A(\beta,m,x)=A(m,x)~.\] 
    Hence 
    \[0\leq A(\beta,m,x)\leq A(m,x).\]
Now we can use \cite[Theorem 6.1]{sstt} to bound $A(m,x)$, yielding the desired result. This proves \Cref{local-infinite}.  
\end{proof}

\section{Non-archimedean estimates}
Our goal in this section is to prove the local non-archimedean estimates \Cref{good-finite}.

Let $\rho:\mathscr{S}\rightarrow\widetilde{\mathcal{M}}_r^{\Sigma}$ be the period map, where $\widetilde{\mathcal{M}}_r^{\Sigma}$ is the toroidal compactification of the integral model of the GSpin Shimura variety associated to $(rL,Q)$. 

\subsection{Good reduction case}
 Let $\mathfrak{P}$ be a prime of good reduction. By the moduli interpretation of $\mathcal{Z}(\beta,m)$, see \cite[Lemma 7.2]{sstt} for a proof, we have: 
 
\[(\mathscr{S}.\mathcal{Z}(\beta,m))=\sum_{n=1}^{\infty}|\{x\, \textrm{lifts to order $n$}, x\in V_\beta(\mathcal{A}_{\mathfrak P}), Q(x)=m\}|~.\]

In particular, 
\[0\leq (\mathscr{S}.\mathcal{Z}(\beta,m))\leq (\mathscr{S}.\mathcal{Z}(m))~,\]
where $\mathcal{Z}(m)\rightarrow \mathcal{M}$ is the special divisor in $\mathcal{M}$

By \cite[Theorem 7.1]{sstt}, we have the estimate: 
\[\sum_{m\in S_{D,X}}(\mathscr{S}.\mathcal{Z}(m))=o(X^{b/2}\log(X)).\]

Hence, combined with the inequality above, we get: 

\[\sum_{m\in S_{D,X}}(\mathscr{S}.\mathcal{Z}(\beta,m))=o(X^{b/2}\log(X))~,\]
which proves   \Cref{good-finite}]
    
\subsection{Bad reduction case: type II}
Let $\mathfrak{P}$ be a prime of bad reduction. The toroidal compactification $\widetilde{\mathcal{M}}_r^{\Sigma}$ has a stratification with two types of boundary components as explained in \cite{tayouboundary}. We will use the results from that paper to analyze the local intersection multiplicities and we focus now on boundary components of type II.
\medskip 

Let $\Upsilon$ be a toroidal stratum representative of type II , $\mathcal{B}^{\Upsilon}$ the corresponding boundary component of type II and we assume in this section that that boundary point $\mathscr{S}(\F_\mathfrak{P})$ lies in $\mathcal{B}^{\Upsilon}(\overline{\mathbb{F}}_p)$.

By \Cref{pull-special-divisors}, we have
\begin{align}\label{inequality-II}    
0\leq (\mathscr{S}.\mathcal{Z}(\beta,m))_{\mathfrak{P}}\leq (\mathscr{S}.\mathcal{Z}(m)))_{\mathfrak{P}}.
\end{align}

Let $D\in \Z_{\geq 1}$. For $X\in \Z_{>0}$, let $S_{D,X}$ denote the set \[\{m\in \Z_{>0}\mid X \leq m<2X,\, \frac{m}{D}\in \Z \cap (\Q^\times)^2,\, (m,N)=1\}.\] 

Then by \cite[Proposition 5.2]{tayouboundary},  we have
\[\sum_{m\in S_{D,X}}(\cY . \cZ(m))_\mathfrak{P}=o(X^{\frac{b+1}{2}}\log X).\]
Combining \Cref{inequality-II} and the previous estimate, we get \Cref{good-finite} in the type II case.

\subsection{Bad reduction case: type III}

Let $(\Xi,\sigma)$ be a toroidal stratum representative of type III where $\sigma$ is a ray. Let $\mathcal{B}^{\Xi,\sigma}$ be the corresponding boundary component of type III and we assume in this section that the boundary point $\mathscr{S}$ lies in $\mathcal{B}^{\Xi,\sigma}(\overline{F}_p)$.

Similarly, we have \Cref{pull-special-divisors} 
\[0\leq (\mathscr{S}.\mathcal{Z}(\beta,m))_\mathfrak{P}\leq (\mathscr{S}.\mathcal{Z}(m))~.\]
Let $D\in \Z_{\geq 1}$ coprime to $N$ and $X\in \Z_{>0}$. Let  $S_{D,X}$ denote the set \[\{m\in \Z_{>0}\mid X \leq m<2X,\, \frac{m}{D}\in \Z \cap (\Q^\times)^2,\,(m,N)=1\}.\] Then we have by \cite[Proposition 5.4]{tayouboundary},
\[\sum_{m\in S_{D,X}}(\cY . \cZ(m))_\fP=o(X^{\frac{b+1}{2}}\log X).\]

Combining the two previous estimates concludes the proof of \Cref{good-finite} in the type III case. 

\begin{remarque}
    As $m\equiv Q(\beta)\pmod{r}$, we see that in order to apply the results proved in \cite{tayouboundary}, we need that $Q(\beta)$ is coprime to $N$.
\end{remarque}

\bibliographystyle{alpha}
\bibliography{bibliographie}

\newcommand{\etalchar}[1]{$^{#1}$}
\begin{thebibliography}{AGHM18}

\bibitem[AGHM17]{agmp-compositio}
Fabrizio Andreatta, Eyal~Z. Goren, Benjamin Howard, and Keerthi Madapusi.
\newblock Height pairings on orthogonal {S}himura varieties.
\newblock {\em Compos. Math.}, 153(3):474--534, 2017.

\bibitem[AGHM18]{agmp-annals}
Fabrizio Andreatta, Eyal~Z. Goren, Benjamin Howard, and Keerthi Madapusi.
\newblock Faltings heights of abelian varieties with complex multiplication.
\newblock {\em Ann. of Math. (2)}, 187(2):391--531, 2018.

\bibitem[AGV{\etalchar{+}}72]{grothendiecksga4}
Michael Artin, Alexander Grothendieck, J.~L. Verdier, N.~Bourbaki, P.~Deligne,
  and Bernard Saint-Donat, editors.
\newblock {\em S{\'e}minaire de g{\'e}om{\'e}trie alg{\'e}brique du
  {Bois}-{Marie} 1963--1964. {Th{\'e}orie} des topos et cohomologie {\'e}tale
  des sch{\'e}mas. ({SGA} 4). {Un} s{\'e}minaire dirig{\'e} par {M}. {Artin},
  {A}. {Grothendieck}, {J}. {L}. {Verdier}. {Avec} la collaboration de {N}.
  {Bourbaki}, {P}. {Deligne}, {B}. {Saint}-{Donat}. {Tome} 1: {Th{\'e}orie} des
  topos. {Expos{\'e}s} {I} {\`a} {IV}. 2e {\'e}d.}, volume 269 of {\em Lect.
  Notes Math.}
\newblock Springer, Cham, 1972.

\bibitem[AMRT10]{amrt}
Avner {Ash}, David {Mumford}, Michael {Rapoport}, and Yung-Sheng {Tai}.
\newblock {\em {Smooth compactifications of locally symmetric varieties. With
  the collaboration of Peter Scholze}}.
\newblock Cambridge: Cambridge University Press, 2010.

\bibitem[AS64]{handbook}
Milton Abramowitz and Irene~A. Stegun.
\newblock {\em Handbook of mathematical functions with formulas, graphs, and
  mathematical tables}, volume~55 of {\em National Bureau of Standards Applied
  Mathematics Series}.
\newblock For sale by the Superintendent of Documents, U.S. Government Printing
  Office, Washington, D.C., 1964.

\bibitem[BBK07]{burgos}
Jan~H. {Bruinier}, Jos\'e~I. {Burgos Gil}, and Ulf {K\"uhn}.
\newblock {Borcherds products and arithmetic intersection theory on Hilbert
  modular surfaces}.
\newblock {\em {Duke Math. J.}}, 139(1):1--88, 2007.

\bibitem[BK03]{bruinierintegrals}
Jan~Hendrik Bruinier and Ulf K\"uhn.
\newblock Integrals of automorphic {G}reen's functions associated to {H}eegner
  divisors.
\newblock {\em Int. Math. Res. Not.}, 31:1687--1729, 2003.

\bibitem[Bor98]{borcherds-inventiones-automorphic}
Richard~E. Borcherds.
\newblock Automorphic forms with singularities on {G}rassmannians.
\newblock {\em Invent. Math.}, 132(3):491--562, 1998.

\bibitem[Bru02]{bruinier}
Jan~H. Bruinier.
\newblock {\em Borcherds products on {O}(2, {$l$}) and {C}hern classes of
  {H}eegner divisors}, volume 1780 of {\em Lecture Notes in Mathematics}.
\newblock Springer-Verlag, Berlin, 2002.

\bibitem[BZ21]{bruinierzemel}
Jan~Hendrik {Bruinier} and Shaul {Zemel}.
\newblock {Special Cycles on Toroidal Compactifications of Orthogonal Shimura
  Varieties}.
\newblock {\em Math. Z.}, 2021.

\bibitem[Cha14]{charles-Picard-number}
Fran{\c{c}}ois Charles.
\newblock On the {P}icard number of {K}3 surfaces over number fields.
\newblock {\em Algebra Number Theory}, 8(1):1--17, 2014.

\bibitem[Cha18]{charles-exceptional-isogenies}
François Charles.
\newblock Exceptional isogenies between reductions of pairs of elliptic curves.
\newblock {\em Duke Math. J.}, 167(11):2039--2072, 08 2018.

\bibitem[EGT23]{tayoumock}
Philip {Engel}, Fran{\c{c}}ois {Greer}, and Salim {Tayou}.
\newblock {Mixed mock modularity of special divisors}.
\newblock {\em arXiv e-prints}, page arXiv:2301.05982, January 2023.

\bibitem[FHVA22]{frei-hassett-alvarado}
Sarah Frei, Brendan Hassett, and Anthony V{\'a}rilly-Alvarado.
\newblock Reduction of {Brauer} classes on {K3} surfaces, rationality and
  derived equivalence.
\newblock {\em J. Reine Angew. Math.}, 792:289--305, 2022.

\bibitem[GS90]{gilletsoule}
Henri Gillet and Christophe Soul\'{e}.
\newblock Arithmetic intersection theory.
\newblock {\em Inst. Hautes \'{E}tudes Sci. Publ. Math.}, 72:93--174 (1991),
  1990.

\bibitem[Han04]{hanke}
Jonathan Hanke.
\newblock Local densities and explicit bounds for representability by a
  quadratric form.
\newblock {\em Duke Math. J.}, 124(2):351--388, 2004.

\bibitem[HM20]{howardmadapusi}
Benjamin Howard and Keerthi Madapusi.
\newblock Arithmetic of {B}orcherds products.
\newblock {\em Ast\'{e}risque}, (421, Diviseurs arithm\'{e}tiques sur les
  vari\'{e}t\'{e}s orthogonales et unitaires de Shimura):187--297, 2020.

\bibitem[HM22]{howardmadapusi-2}
Benjamin {Howard} and Keerthi {Madapusi}.
\newblock {Kudla's modularity conjecture on integral models of orthogonal
  Shimura varieties}.
\newblock {\em arXiv e-prints}, page arXiv:2211.05108, November 2022.

\bibitem[Huy16]{huybrechts}
Daniel Huybrechts.
\newblock {\em Lectures on {K}3 surfaces}, volume 158 of {\em Cambridge Studies
  in Advanced Mathematics}.
\newblock Cambridge University Press, Cambridge, 2016.

\bibitem[IIK21]{ito-ito-koshikawa}
Kazuhiro Ito, Tetsushi Ito, and Teruhisa Koshikawa.
\newblock {CM} liftings of {{\(K3\)}} surfaces over finite fields and their
  applications to the {Tate} conjecture.
\newblock {\em Forum Math. Sigma}, 9:70, 2021.
\newblock Id/No e29.

\bibitem[Kis10]{kisin}
Mark Kisin.
\newblock Integral models for {S}himura varieties of abelian type.
\newblock {\em J. Amer. Math. Soc.}, 23(4):967--1012, 2010.

\bibitem[KM16]{kimmadapusipera}
Wansu Kim and Keerthi Madapusi.
\newblock 2-adic integral canonical models.
\newblock {\em Forum Math. Sigma}, 4:e28, 34, 2016.

\bibitem[Mad15]{madapusiperatate}
Keerthi Madapusi.
\newblock The {T}ate conjecture for {K}3 surfaces in odd characteristic.
\newblock {\em Invent. Math.}, 201(2):625--668, 2015.

\bibitem[Mad16]{madapusiintegral}
Keerthi Madapusi.
\newblock Integral canonical models for spin {S}himura varieties.
\newblock {\em Compos. Math.}, 152(4):769--824, 2016.

\bibitem[Mad19]{madapusitor}
Keerthi Madapusi.
\newblock {Toroidal compactifications of integral models of Shimura varieties
  of Hodge type.}
\newblock {\em {Ann. Sci. \'Ec. Norm. Sup\'er. (4)}}, 52(2):393--514, 2019.

\bibitem[McG03]{mcgraw}
William~J. McGraw.
\newblock The rationality of vector valued modular forms associated with the
  {W}eil representation.
\newblock {\em Math. Ann.}, 326(1):105--122, 2003.

\bibitem[MST22a]{maulik-shankar-tang-K3}
Davesh {Maulik}, Ananth~N. {Shankar}, and Yunqing {Tang}.
\newblock {Picard ranks of K3 surfaces over function fields and the Hecke orbit
  conjecture}.
\newblock {\em Inventiones Mathematicae}, January 2022.

\bibitem[MST22b]{maulik-shankar-tang}
Davesh Maulik, Ananth~N. Shankar, and Yunqing Tang.
\newblock Reductions of abelian surfaces over global function fields.
\newblock {\em Compos. Math.}, 158(4):893--950, 2022.

\bibitem[Pin90]{pink}
Richard Pink.
\newblock {\em Arithmetical compactification of mixed {S}himura varieties},
  volume 209 of {\em Bonner Mathematische Schriften [Bonn Mathematical
  Publications]}.
\newblock Universit\"{a}t Bonn, Mathematisches Institut, Bonn, 1990.
\newblock Dissertation, Rheinische Friedrich-Wilhelms-Universit\"{a}t Bonn,
  Bonn, 1989.

\bibitem[SSTT22]{sstt}
Ananth~N. Shankar, Arul Shankar, Yunqing Tang, and Salim Tayou.
\newblock Exceptional jumps of {Picard} ranks of reductions of {K3} surfaces
  over number fields.
\newblock {\em Forum Math. Pi}, 10:49, 2022.
\newblock Id/No e21.

\bibitem[ST20]{shankar-tang}
Ananth~N. {Shankar} and Yunqing {Tang}.
\newblock {Exceptional splitting of reductions of abelian surfaces}.
\newblock {\em {Duke Math. J.}}, 169(3):397--434, 2020.

\bibitem[{Sta}23]{stacks-project}
The {Stacks project authors}.
\newblock The stacks project.
\newblock \url{https://stacks.math.columbia.edu}, 2023.

\bibitem[{Tay}20]{tayouequi}
Salim {Tayou}.
\newblock {On the equidistribution of some Hodge loci}.
\newblock {\em {J. Reine Angew. Math.}}, 762:167--194, 2020.

\bibitem[{Tay}24]{tayouboundary}
Salim {Tayou}.
\newblock {Picard rank jumps for K3 surfaces with bad reduction}.
\newblock {\em Algebra \& Number Theory}, 2024.

\bibitem[TT23]{tayoutholozan}
Salim Tayou and Nicolas Tholozan.
\newblock Equidistribution of {Hodge} loci. {II}.
\newblock {\em Compos. Math.}, 159(1):1--52, 2023.

\bibitem[Voi02]{voisin}
C.~Voisin.
\newblock {\em Th{\'e}orie de Hodge et g{\'e}om{\'e}trie alg{\'e}brique
  complexe}.
\newblock Collection SMF. Soci{\'e}t{\'e} Math{\'e}matique de France, 2002.

\end{thebibliography}

\end{document}